\documentclass[reqno,11pt]{amsart}
 \usepackage{bbm}
\usepackage{url}
\usepackage{soul}
\usepackage{epstopdf}
\usepackage{algorithm}
\usepackage{algpseudocode}
\usepackage[a4paper,bindingoffset=0.5cm,left=2cm,right=2cm,top=2.5cm,bottom=2cm,footskip=.8cm]{geometry}

\usepackage[T1]{fontenc}
\usepackage[]{natbib}
\usepackage{hyperref}
\usepackage{caption}
\usepackage{bookmark}
\usepackage{pgfplots}
\usepackage{grffile}
\newlength\figureheight
\newlength\figurewidth
\pgfplotsset{%
        tick label style={font=\scriptsize},
        label style={font=\footnotesize},
        legend style={font=\footnotesize},
        every axis plot/.append style={very thick}
}
\newcommand\bitsmall{\fontsize{9.5}{11}\selectfont}

 \renewcommand{\leq}{\leqslant}
\renewcommand{\geq}{\geqslant}

\allowdisplaybreaks

\makeatletter
\renewcommand{\fnum@figure}[1]{\textbf{\figurename~\thefigure}. }
\renewcommand{\fnum@table}[1]{\textbf{\tablename~\thetable}. }
\makeatother

\usepackage[utf8]{inputenc}
\usepackage{graphicx}
\usepackage{subfig}
\usepackage{type1cm}        
\usepackage{multicol}        
\usepackage[bottom]{footmisc}

\usepackage{newtxtext}       %
\usepackage{newtxmath}       

\newcommand{\vb}{\vspace{3.2mm}}

\usepackage[british,UKenglish,USenglish,english,american,dutch]{babel} 
\usepackage{mathrsfs}
\usepackage{lipsum}
\usepackage{verbatim}
\usepackage{amsmath, amsfonts, amsthm,mathtools}
\usepackage{syntonly}
\usepackage{float}
\usepackage{subcaption}
\usepackage{enumitem}

\usepackage{bbm}
\usepackage{eufrak}
\usepackage{tikz-cd} 
\usepackage{adjustbox}

\newtheorem{lemma}{Lemma}
\newtheorem{corollary}{Corollary}
\newtheorem{assumptions}{Assumptions}
\newtheorem{assumption}[assumptions]{Assumption}

\newtheorem{innercustomthm}{Assumptions}
\newenvironment{customthm}[1]
  {\renewcommand\theinnercustomthm{#1}\innercustomthm}
  {\endinnercustomthm}
\newtheorem{theorem}{Theorem}
\newtheorem{remark}{Remark}

\usepackage{hyperref}
\usepackage{multirow}
\usepackage{gensymb}

\begin{document}
\selectlanguage{USenglish}

        \title[Asymptotically distribution-free GoF testing for point processes]{Asymptotically distribution-free goodness-of-fit testing for point processes}       
\author{Justin R. Baars, S. Umut Can, and Roger J.~A. Laeven}
       
        \begin{abstract}
Consider an observed path from a multivariate temporal point process $N$ with law $\mathbb P$ on the time interval $[0,T]$. 
To test the null hypothesis that $\mathbb P$ belongs to a given parametric family, we construct a compensated empirical process to which we apply a suitable innovation martingale transformation. 
We prove that the resulting \emph{test process} converges weakly to a standard Wiener process. 
Consequently, taking a testing functional of this process yields an asymptotically distribution-free goodness-of-fit test for parametric point processes. 
Furthermore, for standard test statistics based on the increments of this test process, we establish consistency under alternative hypotheses. 
We assess the performance of the proposed testing procedure through a Monte Carlo simulation study and illustrate its practical utility with two real-data examples.

\vb

\vb

\noindent
{\sc Email.} \url{j.r.baars@uva.nl},  \url{s.u.can@uva.nl}, \url{r.j.a.laeven@uva.nl}.

\vb

\noindent
\textit{Date}: \today.

        \end{abstract}

        \maketitle

      \section{Introduction}\label{section introduction}
Parametric point processes are widely used 
to represent the occurrence and evolution of real-world events over time in fields such as seismology, epidemiology, neuroscience, finance, social network analysis, and criminology.
Consider a $d$-variate 
point process $N=(N^{(1)},\ldots,N^{(d)})$ indexed by time $t\in\mathbb R$. 
Suppose that the law $\mathbb P$ of $N$ is modeled through a 
set of probability laws 
$\mathscr F_\Theta:=\{\mathbb P_\theta:\theta\in\Theta\}$, where $\Theta\subset\mathbb R^m$ is a finite-dimensional parameter space. 
If $\mathbb P\in\mathscr F_\Theta$, it holds that $N\sim \mathbb P_{\theta_0}$ for some $\theta_0\in\Theta$. 
In practice, $\theta_0$ is typically unknown, and has to be estimated using a time-ordered observation $\omega_T$
of the point process, i.e., an observed path from $N$ on the time interval $[0,T]$, leading to an estimated parameter vector $\hat\theta_T$. 
If, for example, maximum likelihood estimation is employed, the statistical properties of $\hat\theta_T$ are well-established \citep{Ogata, Puri, MLEinfinitedimensional}.

This paper is concerned with goodness-of-fit testing for temporal point processes. 
Given an observation $\omega_T$, we aim to test the parametric null hypothesis $H_0: N\sim \mathbb P_{\theta}$ for some $\theta\in\Theta$ against the alternative $H_1: N\not\sim \mathbb P_{\theta}$ for all $\theta\in\Theta$. 
For temporal point processes, a common approach is to apply to each coordinate $k\in[d]:=\{1,\ldots,d\}$ a time transformation $t\mapsto\Lambda^{(k)}(t)$, where $\Lambda^{(k)}(t)=\int_0^t\lambda^{(k)}(s)\, \mathrm ds$ is the time-integrated conditional intensity of $N^{(k)}$, i.e., the \emph{compensator}. 
According to the random time change theorem \citep[Theorem 7.4.VI]{DaleyVereJones}, under $\mathbb P_{\theta_0}$ it holds that through 
the simultaneous random time transformations $t\mapsto\Lambda_{\theta_0}^{(k)}(t)$, the $d$-variate point process $N$ is transformed into a $d$-variate Poisson process of unit rate. 
Since $\theta_0$ is unknown, the goodness-of-fit of a candidate model is then typically assessed by testing whether the transformed interarrival times $\Lambda_{\hat\theta_T}^{(k)}(t_{i}^{(k)})-\Lambda_{\hat\theta_T}^{(k)}(t_{i-1}^{(k)})$ are standard exponential, where $\hat\theta_T$ is an estimator of $\theta_0$, following the seminal work of \citet{Ogata2}.
In such tests, the estimation error $|\hat\theta_T-\theta_0|$, which causes a discrepancy between $\Lambda_{\hat\theta_T}$ and $\Lambda_{\theta_0}$, is often ignored;
see \citet{Ogata2}, Section~3.3, and the references in Section~\ref{sectionRTC}.

Ignoring estimation uncertainty can have important adverse consequences. 
The asymptotic distribution of test statistics under the null hypothesis typically depends on the specific parametric family $\mathscr F_\Theta$ that is tested for and on the true parameter $\theta_0$. 
That is, the test statistics are not asymptotically distribution-free, and there is generally no justification for ignoring uncertainty induced by the estimation error $|\hat\theta_T-\theta_0|$. 
In fact, even if the data is generated under the law $\mathbb P_{\theta_0}$, the estimator $\hat{\theta}_T$ will empirically often fit the observed data better than the true parameter $\theta_0$ (surely, when using MLE). 
Hence, for a test statistic $S(\omega_T, \theta)$ depending on both the observation $\omega_T$ and the parameter vector $\theta$, it is natural to expect that under $\mathbb{P}_{\theta_0}$, the value of $S(\omega_T, \hat{\theta}_T)$ will appear ``less unlikely'' than $S(\omega_T, \theta_0)$. 
Neglecting estimation uncertainty then leads to undersized tests with reduced power.\footnote{
In applied work, it is common to estimate a parametric model from the data and then use a goodness-of-fit test attempting to demonstrate that the model fits well by showing that the test fails to reject the null hypothesis. 
Ignoring estimation uncertainty in this context then leads to overly conservative tests (i.e., with too low power), potentially giving a false sense of confidence in an unfit model.}

In the past decades, several other as well as related statistical procedures for residual analysis of \mbox{(spatio-)} temporal point processes have been proposed, including most noticeably random thinning \citep{methods2,Schoenberg03}, random superposition \citep{methods3}, super-thinning \citep{methods1}, and random-time-change-based methods \citep{methods4, methods5}, each with their own advantages. 
These foundational methods, like the approach described above, are based on results that hold under the true distribution $\mathbb P_{\theta_0}$. 
A common difficulty of the resulting testing procedures is their dependence on $\mathscr F_\Theta$ and the true but unknown parameter $\theta_0$. 
Consequently, to account for estimation uncertainty, specialized resampling methods (e.g., bootstrap procedures) are required to 
approximate the asymptotic distribution of the test statistics under the null hypothesis, 
for each parametric model. 

In recent years, substantial efforts have been made to address this problem. 
For example, asymptotically correct goodness-of-fit tests have been developed for the specific cases of self-exciting temporal point processes \citep{ElAroui} and exponential Hawkes processes \citep{Bonnet}. 
Additionally, sparsity tests for the connectivity matrices of multivariate Hawkes processes have been introduced \citep{Lotz}. 
Important progress has also been made by \citet{ReynaudBouret}, who mitigate the bias caused by using the plugged-in estimated parameter in scenarios with a large number of (sub)sample paths. 
Furthermore, \citet{KlingVetter} have proposed a bootstrap-based goodness-of-fit testing procedure for self-exciting point processes; 
their procedure is empirically applicable to a broad class of self-exciting models, while asymptotic consistency in an infill-asymptotic setting is established in the special case of inhomogeneous Poisson processes.
Moreover, asymptotically correct goodness-of-fit tests have been proposed for cases where the temporal dynamics are partially known \textit{a priori} \citep{Clinet, Richards}. 
Nevertheless, none of these works provides an asymptotically distribution-free goodness-of-fit testing procedure for a single observed path from a general temporal point process.

In this paper, we 
develop an asymptotically distribution-free goodness-of-fit testing procedure for a general class of multivariate temporal point processes. 
Specifically, based on the observation $\omega_T$, we construct a \emph{test process} on an interval $[0, \tau]$ that, as we establish, converges weakly to a standard Wiener process, as $T \to \infty$. 
Consequently, the limiting process has a distribution that is independent of both the null hypothesis under scrutiny and the true parameter $\theta_0$. 
Appropriate functionals of this test process can then be employed as test statistics. 
Since the weak limit of the test process is a standard Wiener process, the asymptotic distributions of these functionals will also be independent of the parametric model class $\mathscr F_\Theta$ that is tested for and $\theta_0$. 
This means that their critical values only need to be tabulated once, making them applicable across a wide range of testing problems --- an important advantage.
We also establish consistency under general alternative hypotheses, for standard test statistics.

To derive the asymptotically distribution-free test process, we first establish a functional central limit theorem (FCLT) for the point process $N\sim \mathbb P_{\theta_0}$ compensated by $\Lambda_{\hat\theta_T}$. 
The resulting limit is the sum of a Wiener process and a linear drift term with a random coefficient, which depends on $\mathscr F_\Theta$ and $\theta_0$. 
We show that the specific structure of this limit enables the application of a suitable \emph{innovation martingale transformation} $\mathscr{T}_{\theta_0}$, originally introduced by \citet{Estate81, Estate88, Estate93}, which transforms the limit into a standard Wiener process. 
Applying the empirical counterpart $\mathscr T_{\hat\theta_T}$ of this transformation yields an empirical process that, as we formally demonstrate, converges to a standard Wiener process as $T \to \infty$.

This paper contributes to the literature by introducing a novel, general framework for goodness-of-fit testing of multivariate temporal point processes belonging to a finite-dimensional parametric family. 
The constructed test process is asymptotically distribution-free. 
At the core of the theoretical development is an FCLT for the transformed compensated empirical process, i.e., the test process, obtained by combining martingale arguments with a careful analysis of the difference between population and empirical compensators, an  application of a suitable innovation martingale transformation along with a thorough dissection of its empirical counterpart. 
In developing this framework, we also establish a Cramér-Wold device for stochastic processes and derive an FCLT for multivariate martingales --- results that are of potential independent interest. 
The proposed tests are based on standard functionals of the test process, are formally shown to remain consistent under general alternatives, and account for estimation uncertainty. 
Compared to existing methods, this leads to substantially improved reliability and greater power. 

We demonstrate these advantages through Monte Carlo simulation studies and two applications to real data.
Our simulation results show that our testing procedure has good size and power properties in a variety of realistic settings. 
They also confirm that tests that ignore estimation uncertainty tend to be undersized and to have substantially lower power. 
In a data analysis, we demonstrate that our approach, when applied to real earthquake data, reveals that the classical temporal ETAS model is less fit for the data than would be concluded on the basis of existing standard tests that ignore estimation uncertainty. 
We also show, using Rocky Mountain Spotted Fever data from Florida, that a family of recursive point processes for epidemic data is rejected when using our testing procedure but not when using existing, less powerful tests ignoring estimation uncertainty. 

The remainder of this paper is organized as follows. 
Section~\ref{section def} introduces the setup and assumptions. 
Section~\ref{Section3} derives the functional central limit theorem for the compensated empirical process after applying an observable innovation martingale transformation.
Section~\ref{SectionGoF} focuses on goodness-of-fit testing, proposing asymptotically distribution-free test statistics and discussing biases in random-time-change-based methods.
Section~\ref{Section4.2} establishes consistency under general alternatives.
Section~\ref{Section5.1} presents simulation results comparing our method to existing approaches.
Section~\ref{Section5.2} illustrates the method’s performance in two real-data applications.
Section~\ref{SectionConclusion} outlines possible extensions.
In the appendices, we provide technical details, proofs, additional simulation and empirical results.\footnote{Reproducibility files (code, data, and report) are available from the authors upon request.}
Appendix~\hyperref[app A]{A} contains technical assumptions for
Section~\ref{section def}; Appendix~\hyperref[app B]{B} contains proofs;
Appendix~\hyperref[app C]{C} contains additional details for Section~\ref{Section5.1}; and
Appendix~\hyperref[app D]{D} contains additional details for Section~\ref{Section5.2}.

\section{Problem setup and assumptions}\label{section def}
Let $N=(N^{(1)},\ldots,N^{(d)})$ be a $d$-variate temporal point process on $\mathbb R$, i.e.,\ a random discrete $[d]$-marked subset of $\mathbb R$. 
A realization of $N$ consists of event times and coordinates, and is denoted by $\omega:=\{\ldots,
(t_{-1},k_{-1}),(t_0,k_0),(t_1,k_1),
\ldots\}\subset \mathbb R\times[d]$, where $t_i<t_{i+1}$, $i\in\mathbb{Z}$, and $t_0<0\leq t_1$. 
The realization on $[0,T]$ is denoted by $\omega_T=\omega\cap([0,T]\times[d])=\{(t_1,k_1),\ldots,(t_{\bar N(T)},k_{\bar N(T)})\}$, where $\bar N(T)=\sum_{k=1}^dN^{(k)}(T)$ represents the number of events in the time interval $[0,T]$. 
We assume that $\omega$ has no limit point and 
denote the law of $N$ by $\mathbb P$. 

Given an observation $\omega_T$ of $N$, our ultimate aim is to test, for a parametric family $\mathscr F_\Theta=\{\mathbb P_\theta:\theta\in\Theta\}$,
the parametric null hypothesis
\begin{equation}\label{H_0 eq}
H_0:N\sim \mathbb P_\theta\text{ for some }\theta\in\Theta,
\end{equation}
where $\Theta\subset\mathbb R^m$ is assumed to be open, convex and bounded. 
In this and the next section, we will suppose that the null hypothesis holds, i.e., $N\sim\mathbb P_{\theta_0}$ for some $\theta_0\in\Theta$. 

We start by imposing some assumptions on the parametric family of point processes $\mathscr F_\Theta$, required throughout.
\begin{assumptions}\label{assA}\phantom.
\begin{enumerate}[label=(\roman*)]
\item Under each $\mathbb P_\theta\in\mathscr F_\Theta$, $N$ is a stationary and ergodic point process, absolutely continuous w.r.t.\ the unit Poisson process on $[0,T]$, for each $T\in(0,\infty)$.
\item Under each $\mathbb P_\theta\in\mathscr F_\Theta$, $N$ is an orderly point process, i.e., $\mathbb P_\theta (\bar N(\delta)\geq2)=o(\delta)$, as $\delta\downarrow0$.
\end{enumerate}
\end{assumptions}

By Assumptions~\ref{assA}, it follows that each $\mathbb P_\theta\in\mathscr F_\Theta$ can be characterized by a conditional intensity function $\lambda_\theta:\mathbb R_+\to\mathbb R_+^d:t\mapsto\lambda_\theta(t)$ of $N\sim\mathbb P_\theta$; see \citet[Proposition~7.3.IV]{DaleyVereJones}. 
We denote the components of $\lambda_\theta$ by $\lambda_\theta^{(k)}$ for $k\in[d]$. 
Given some history $\mathcal H_{0-}$ of $N$ before time $0$, let
$$
\mathcal H_{t-}:=\sigma(N(s):s\in[0,t))\otimes \mathcal H_{0-}.
$$
Then $\lambda_\theta$ can be taken to be any $(\mathcal H_t)_{t\in\mathbb R}$-predictable function such that
\begin{equation}\label{condint def}
\lambda_\theta(t)=\lim_{\Delta t\downarrow0}\frac1{\Delta t}\mathbb P_\theta\left(N[t,t+\Delta t)>0\,|\,\mathcal H_{t-}\right), \qquad t \ge 0.
\end{equation}

The proofs of our main results rely on the following regularity assumptions for the parametric family of conditional intensities $\mathscr{L}_\Theta := \{\lambda_\theta : \theta \in \Theta\}$.
We postpone further discussion of these assumptions to Section~\ref{Section3} and Appendix~\hyperref[app B]{B}.
By stationarity, i.e., Assumption~\ref{assA}(i), Assumptions~\ref{assB} can be stated without reference to $t$.
Here, $C^{1}(\Theta)$ denotes the set of continuously differentiable functions on $\Theta$, and $L^{p}(\mathbb{Q})$ denotes the set of (equivalence classes of) measurable functions $f$  such that $|f|^p$ is $\mathbb Q$-integrable.

\begin{assumptions}\label{assB}\phantom.
\begin{enumerate}[label=(\roman*)]
\item For each $\theta\in\Theta$, $\lambda_\theta$ is $(\mathcal H_t)_{t\in\mathbb R}$-predictable.
\item The model $\mathscr F_\Theta$ is identifiable: $\lambda_{\theta_1}(0)=\lambda_{\theta_2}(0)$ a.s.\ if and only if $\theta_1=\theta_2$.
\item $\lambda_\theta(0)\in C^1(\Theta)$ a.s., with for any $\theta\in\Theta$, $\partial \lambda_\theta^{(k)}(0)/\partial\theta_i\in L^2(\mathbb P_\theta)$ for all $k\in[d]$ and $i\in[m]$.
\item For each $\theta\in\Theta$, $\lambda_\theta(0)>0$ $\mathbb P_\theta$-a.s.
\item For each $\theta\in\Theta$, there exists $\Lambda_0\in L^1(\mathbb P_\theta)$ and a neighborhood $U=U(\theta)\subset\Theta$ of $\theta$ such that for all $\theta'\in U$, $\max_{k\in[d]}|\lambda_{\theta'}^{(k)}(0)|\leq\Lambda_0$, $\mathbb P_\theta$-a.s.
\item Let $D\lambda_{\theta}(t)=(\partial \lambda_\theta^{(k)}(t)/\partial\theta_i)_{k\in[d],i\in[m]}\in\mathbb R^{d\times m}$ be the \emph{total derivative} of $\theta\mapsto\lambda_\theta(t)$. 
Assume that
\begin{equation}\label{alpha def}
\alpha_{\theta_0}:=\lim_{T\to\infty}\frac1T\int_0^TD\lambda_{\theta_0}(t)\, \mathrm dt
\end{equation}
exists $\mathbb P_{\theta_0}$-a.s.
\end{enumerate}
\end{assumptions}

\begin{remark}
Assumption~\ref{assB}(vi) is satisfied if the following condition holds: \emph{there exists $\Lambda_1\in L^1(\mathbb P_{\theta_0})$ and a neighborhood $U=U(\theta_0)\subset\Theta$ of $\theta_0$ such that for all $\theta'\in U$, $\|D\lambda_{\theta'}(0)\|\leq\Lambda_1$},  $\mathbb P_{\theta_0}$-a.s.
This follows by ergodicity (Assumption~\ref{assA}(i)) together with the dominated convergence theorem. 
In this case, using stationarity of $\lambda_{\theta}$, it follows that $\alpha_\theta=\mathbb E_\theta[D\lambda_{\theta}(0)]$.
\end{remark}

Next, suppose that we have some estimator $\hat\theta_T\in\Theta$ for $\theta$ based on the observation $\omega_T$ of $N$. 
Our main results will assume that this estimator satisfies a central limit theorem, as follows.

\begin{assumption}\label{CLTass}
There exists an $m$-variate random vector $Z$ such that for the estimator $\hat\theta_T\in\Theta$, under $\mathbb P_{\theta_0}$ we have
\begin{equation}\label{CLT eq}
\sqrt T(\theta_0-\hat\theta_T)\stackrel{d}\to Z,\quad\text{as }T\to\infty.
\end{equation}
\end{assumption}
Assumption~\ref{CLTass} is satisfied with $Z\sim\mathcal N(0,I(\theta_0)^{-1})$ (where $I(\theta)$ is defined in Assumption~\ref{MLEass}(ii)) when $\hat\theta_T$ is the maximum likelihood estimator and if next to Assumptions~\ref{assA}--\ref{assB} one grants Assumptions~\ref{MLEass}; see Appendix~\hyperref[app A]{A}.
This follows from a multivariate generalization of \citet[Theorem~5]{Ogata}; see, e.g., \citet[Theorem~4]{Puri}.
%
%
%

Assumptions~\ref{assA}, \ref{assB}, and~\ref{MLEass} are satisfied by many families of point processes such as Poisson, Hawkes, Wold, and delayed renewal processes; see, e.g., \citet{Ogata} in the context of maximum likelihood estimation.
Furthermore, our setup and assumptions allow for point processes dependent on covariates, such as marked Hawkes processes.

In Sections~\ref{Section3} and~\ref{Section4.2}, we establish limit results for a certain transformation of the stochastic process $N$. 
Therefore, we need an appropriate framework for stochastic process convergence. 
The transformation of the next sections maps the time interval $[0,T]$ to $[0,1]$, hence it suffices to work on the latter interval. 
On the space $D([0,1],\mathbb R^d)$ of right-continuous functions $f:[0,1]\to\mathbb R^d$, we consider multiple topologies. 
We consider the strong $SU$ and the weak $WU$ uniform topologies and the strong $SJ_1$ and weak $WJ_1$ Skorokhod $J_1$ topologies. 
The strong topologies are defined by treating $\mathbb R^d$ as the range space, whereas the weak topologies use the fact that $D([0,1],\mathbb R^d)=D([0,1],\mathbb R)^d$ in the sense of bijection, and consider a product topology on this space. 
As the name suggests, in general the strong topology is stronger than the weak topology. 
However, as we will prove later, the strong and product uniform topologies coincide in our setting, i.e., $SU=WU$. 
We refer to \citet{Whitt} for more details on the respective spaces and topologies.

\section{Asymptotically distribution-free test process}\label{Section3}
In this section, we construct an asymptotically distribution-free test process from a suitable transformation of the process $N$ under $\mathbb P_{\theta_0}$.
The limit is in particular independent of the parametric family $\mathscr F_\Theta$ as well as the true parameter $\theta_0$. 
As indicated in the Introduction, this enables us to overcome an important limitation of goodness-of-fit tests for point processes known from the literature. 

We start by considering the $d$-dimensional compensated empirical process
\begin{equation}\label{empiricalprocess}
\hat\eta^{(T)}:u\mapsto\frac1{\sqrt T}\left(N(uT)-\int_0^{uT}\lambda_{\hat\theta_T}(s)\, \mathrm ds\right),
\end{equation}
for $u\in[0,1]$,
which we decompose into
\begin{equation}\label{empiricalprocess decomposition}
\frac1{\sqrt T}\left(N(uT)-\int_0^{uT}\lambda_{\theta_0}(s)\ \mathrm ds\right)+\frac1{\sqrt T}\left(\int_0^{uT}\lambda_{\theta_0}(s)\ \mathrm ds-\int_0^{uT}\lambda_{\hat\theta_T}(s)\ \mathrm ds\right).
\end{equation}
Our development then comprises four steps.
\begin{enumerate}
\item[(1.)] The first step consists of developing a martingale functional central limit theorem (FCLT) for the first term of (\ref{empiricalprocess decomposition}), in order to establish that a Gaussian limit arises.
\item[(2.)] We next prove that the second term of (\ref{empiricalprocess decomposition}) converges weakly to a deterministic function of $u$ multiplied by some random variable (which is, in particular, independent of $u$).
\item[(3.)] Steps 1 and 2 yield a weak limit $\eta_{\theta_0}$ of (\ref{empiricalprocess decomposition}) that depends on the true parameter $\theta_0$ and the family $\mathscr F_\Theta$.
We show that the particular shape of the limit allows us to apply an \emph{innovation martingale transformation} $\mathscr T_{\theta_0}$,
first discussed by \citet{Estate81, Estate88, Estate93}. 
We demonstrate that, after a suitable transformation, $\mathscr T_{\theta_0}(\eta_{\theta_0})$ is a standard Wiener process. 
This transformation $\mathscr T_{\theta_0}$ still depends on the true parameter $\theta_0$.
\item[(4.)] Finally, we establish that the difference between $\mathscr T_{\theta_0}(\eta_{\theta_0})$ and its empirical counterpart $\mathscr T_{\hat\theta_T}(\hat\eta^{(T)})$ becomes `small' for large $T$, and hence that $\mathscr T_{\hat\theta_T}(\hat\eta^{(T)})$ converges weakly to a standard Wiener process.
\end{enumerate}

As we will see, due to the orthogonal structure of the martingales underlying a multivariate point process, with some work, the analysis in the \emph{multi}variate case can be reduced to that of the \emph{uni}variate case. 
All the proofs for this section can be found in Appendix~\hyperref[app B]{B}.

\subsection{FCLT for the compensated point process}\label{sec3.1}
The next result establishes Step~1.

\begin{theorem}\label{step1thm}
Grant Assumptions~\ref{assA}.
Let $\lambda_{\theta_0}$ be the conditional intensity of the point process $N\sim\mathbb P_{\theta_0}\in\mathscr F_\Theta$. Then it holds that 
\begin{equation}
\left(\frac1{\sqrt T}\left(N(uT)-\int_0^{uT}\lambda_{\theta_0}(s)\ \mathrm ds\right)\right)_{u\in[0,1]}\stackrel d\longrightarrow (W_{\theta_0}(u))_{u\in[0,1]},
\end{equation}
as $T\to\infty$,
weakly on $D([0,1],\mathbb R^d)$ equipped with the strong Skorokhod $J_1$ topology, where $W_{\theta_0}$ is a zero-mean $d$-variate Wiener process with covariance 
$$u\cdot\mathrm{diag}\left(\mathbb E_{\theta_0}[N^{(1)}(1)],\ldots,\mathbb E_{\theta_0}[N^{(d)}(1)]\right).$$
\end{theorem}
The proof of Theorem~\ref{step1thm} relies on a Cramér-Wold device for stochastic processes --- instead of, as usual, for multi-dimensional random vectors --- that we establish, given as Lemma~\ref{cramerwold} in Appendix~\hyperref[app B]{B}, and an FCLT for \emph{multi}variate martingales, given as Lemma~\ref{FCLTmartingales} in Appendix~\hyperref[app B]{B}.
Both results are of potential independent interest.

\subsection{Limit theorem for the compensators}\label{sec3.2}

For Step~2 of our procedure, we analyze the asymptotic behavior of the second term of (\ref{empiricalprocess decomposition}).

\begin{theorem}\label{step2thm}
Grant Assumptions~\ref{assA}--\ref{CLTass}. 
Let $\alpha_{\theta_0}\in\mathbb R^{d\times m}$ be as in Assumption~\ref{assB}(vi) and let $Z$ be as in Assumption~\ref{CLTass}. 
Then, under the law $\mathbb P_{\theta_0}\in\mathscr F_\Theta$,
\begin{equation}\label{step2limit}
\left(\frac1{\sqrt T}\int_0^{uT}\left(\lambda_{\theta_0}(s)-\lambda_{\hat\theta_T}(s)\right)\, \mathrm ds\right)_{u\in[0,1]}\stackrel d\longrightarrow (u\alpha_{\theta_0}Z)_{u\in[0,1]},
\end{equation}
as $T\to\infty$, weakly on $D([0,1],\mathbb R^d)$ equipped with the strong Skorokhod $J_1$ topology. 
\end{theorem}

\begin{remark}
Since the limit $u\mapsto u\alpha_{\theta_0}Z$ is continuous, it follows that \eqref{step2limit} even holds on $D([0,1],\mathbb R^d)$ equipped with the uniform topology.
\end{remark}

Combining Theorem \ref{step1thm} with Theorem \ref{step2thm}, we obtain the following limit result  for the empirical process $\hat\eta^{(T)}$ defined in (\ref{empiricalprocess}).

\begin{corollary}\label{step1+2}
Grant Assumptions \ref{assA}--\ref{CLTass}. 
Let $W_{\theta_0}$ be a zero-mean $d$-variate Wiener process with covariance 
$$u\cdot\mathrm{diag}\left(\mathbb E_{\theta_0}[N^{(1)}(1)],\ldots,\mathbb E_{\theta_0}[N^{(d)}(1)]\right).$$
Then, under the law $\mathbb P_{\theta_0}\in\mathscr F_\Theta$, 
\begin{align}
\begin{split}\label{step1+2 eq}
\left(\frac1{\sqrt T}\left(N(uT)-\int_0^{uT}\lambda_{\hat\theta_T}(s)\, \mathrm ds\right)\right)_{u\in[0,1]}\stackrel d\longrightarrow &\ (W_{\theta_0}(u)+u\alpha_{\theta_0}Z)_{u\in[0,1]}\\
&=:(\eta_{\theta_0}(u))_{u\in[0,1]},
\end{split}\end{align}
as $T\to\infty$, weakly on $D([0,1],\mathbb R^d)$ equipped with the strong Skorokhod $J_1$ topology.
\end{corollary}

\subsection{Transforming $\eta_{\theta_0}$ into a standard Wiener process}\label{sec3.3}
The limiting result (\ref{step1+2 eq}) derived in Step 2 still depends on the true but unknown parameter $\theta_0$, through the covariance of $W_{\theta_0}$ and the linear stochastic drift term $u\alpha_{\theta_0}Z$. 
Due to this dependency on the true parameter, we cannot directly use this limit to construct an asymptotically distribution-free testing procedure for all $\theta \in \Theta$, causing the theory to become `fragmented' --- requiring a separate treatment for each null model. 
To address this issue, we apply a suitable \emph{innovation martingale transform} to the limit process $\eta_{\theta_0}$, which results in a new limit process that is independent of the true parameter $\theta_0$ and of the specific family $\mathscr F_\Theta$ we are testing for. 

Variants of innovation martingale transformations have been applied to various statistical problems in the literature over the last couple of decades; see \citet{IMT7, 
IMT9, IMT5, IMT6, IMT3, IMT4, IMT1, IMT2, Can15}.  
An innovation martingale transformation has also been applied to point processes involving populations of size $n$ over fixed time horizons $[0,T]$, where an asymptotically distribution-free test process has been derived in the limit as $n\to\infty$ \citep{Andersen, SunCounting}. 
However, that large-$n$ setup differs fundamentally from the large-$T$ setup considered in this work. 
In the large-$n$ regime, one considers a system consisting of many interacting particles over a fixed time horizon. 
The limiting behavior as $n\to\infty$ is typically governed by an LLN or CLT, where randomness is averaged out due to aggregation across the population. 
The resulting asymptotics are then driven by scaling properties of the population-level process.  
By contrast, the large-$T$ regime concerns a single realization of a process evolving over an increasingly long time horizon. Here, the asymptotics are dictated by the temporal structure of the process rather than by population averaging. 
Growth in $T$ does not inherently reduce stochasticity but instead reveals long-term dependencies and fluctuation behavior that are absent in the large-$n$ setting. 
As a result, techniques that rely on concentration due to population size become futile, 
and intrinsically different probabilistic approaches are required to understand the limiting behavior.

We now state our transformation result. 

\begin{theorem}\label{step3thm}
Grant Assumptions~\ref{assA}--\ref{CLTass}. 
Let $\mu_{\theta_0}^{(k)}:=\mathbb E_{\theta_0}[N^{(k)}(1)]$. For $\theta\in\Theta$, let $\mathscr T_\theta:D([0,1],\mathbb R^d)\to D([0,1],\mathbb R^d)$ be the transformation
\begin{equation}\label{transformation eq}
\mathscr T_\theta(\eta)(B):=\left\{\frac1{\sqrt{\mu_{\theta}^{(k)}}}\left(\eta^{(k)}(B)-\int_B\frac{\eta^{(k)}(1)-\eta^{(k)}(v)}{1-v}\, \mathrm dv\right)\right\}_{k\in[d]},\quad B\in\mathcal B[0,1],
\end{equation}
with $\mathcal B[0,1]$ denoting the set of Borel subsets of $[0,1]$. 
Then, under the law $\mathbb P_{\theta_0}\in\mathscr F_\Theta$,  $\mathscr T_{\theta_0}(\eta_{\theta_0})$ is a standard $d$-dimensional Wiener process, where $\eta_{\theta_0}$ is defined in (\ref{step1+2 eq}).
\end{theorem}

\begin{remark}
The transformation in \eqref{transformation eq} induces a dimension-wise annihilation of $u\alpha_{\theta_0}Z$ in \eqref{step1+2 eq} as well as a standardization by $\sqrt{\mu_{\theta}^{(k)}}$ resulting in a standard $d$-dimensional Wiener process.
The standard nature of this limiting process will facilitate asymptotically distribution-free testing.
\end{remark}

\subsection{Replacing $\theta_0$ by $\hat\theta_T$}\label{sec3.4}
Although the limit process from Theorem~\ref{step3thm} does not depend on the true parameter $\theta_0$, the transformation $\mathscr T_{\theta_0}$ and the process $\eta_{\theta_0}$ still do. 
Hence, in practical applications $\mathscr T_{\theta_0}(\eta_{\theta_0})$ cannot be used. 
Instead, one can replace $\mathscr T_{\theta_0}(\eta_{\theta_0})$ by $\mathscr T_{\hat\theta_T}(\hat\eta^{(T)})$, where $\hat\eta^{(T)}$ is the empirical process defined in (\ref{empiricalprocess}). 
As the final Step~4, we show that the difference between $\mathscr T_{\hat\theta_T}(\hat\eta^{(T)})$ and $\mathscr T_{\theta_0}(\eta_{\theta_0})$ vanishes in the limit, as $T\to\infty$. 
Since $\mathscr T_{\theta_0}(\eta_{\theta_0})$ is a standard $d$-dimensional Wiener process, this provides us with an asymptotically distribution-free limit process. 
Also, the process $\mathscr T_{\hat\theta_T}(\hat\eta^{(T)})$  can be constructed explicitly without knowledge of $\theta_0$.

We now state our convergence result for $\mathscr T_{\hat\theta_T}(\hat\eta^{(T)})$.
\begin{theorem}\label{step4thm}
Grant Assumptions~\ref{assA}--\ref{CLTass}. 
Select some $\tau\in(0,1)$. 
Then, under the law $\mathbb P_{\theta_0}\in\mathscr F_\Theta$,  as $T\to\infty$, $\mathscr T_{\hat\theta_T}(\hat\eta^{(T)})$ converges weakly on $D([0,\tau],\mathbb R^d)$ equipped with the strong Skorokhod $J_1$ topology to a standard Wiener process.
\end{theorem}

\begin{remark}
Theorem~\ref{step4thm} may be compared to Corollary~\ref{step1+2}.
Both results establish the asymptotic behavior of an empirical process constructed from $\omega_{T}$.
Whereas $\hat\eta^{(T)}$ has a weak limit depending on $\mathscr F_\Theta$ and $\theta_{0}$, $\mathscr T_{\hat\theta_T}(\hat\eta^{(T)})$ is asymptotically distribution-free.
\end{remark}

\section{Goodness-of-fit testing procedures}\label{SectionGoF}
Recall that, given a sample $\omega_T$ of a point process $N$, we aim to test the parametric null hypothesis \eqref{H_0 eq}.
In practical situations, the sample $\omega_T$ is used to compute an estimator $\hat\theta_T$ for $\theta$, after which one attempts to decide whether $\mathscr F_\Theta=\{\mathbb P_\theta:\theta\in\Theta\}$ is a suitable parametric family of distributions for the data. 
In Section~\ref{Section4.1}, we show how to employ Theorem~\ref{step4thm} to construct asymptotically distribution-free goodness-of-fit tests for testing \eqref{H_0 eq}. 
We also briefly discuss a related, ``na\"ive'' testing procedure that ignores estimation uncertainty.
In Section~\ref{sectionRTC}, we discuss conventional random-time-change-based testing. 
The testing procedures of this section will be implemented, evaluated and applied in Sections~\ref{Section5.1}--\ref{Section5.2}.
In Section~\ref{Section4.2}, we theoretically analyze consistency under alternative hypotheses.
\subsection{Asymptotically distribution-free and na\"ive goodness-of-fit testing}\label{Section4.1}
In Theorem~\ref{step4thm} we have established that, for any $\tau\in(0,1)$, the process $\mathscr T_{\hat\theta_T}(\hat\eta^{(T)})$ converges weakly on $D([0,\tau],\mathbb R^d)$ to a standard Wiener process $W$ under $H_0$. 
Now suppose that $\mathscr K$ is a test statistic (i.e., functional) that depends continuously on $\mathbb R^d$-valued stochastic processes on $[0,\tau]$. 
Then $\mathscr K(\mathscr T_{\hat\theta_T}(\hat\eta^{(T)}))$ will converge in distribution to $\mathscr K(W)$.
The limit process $W$ from Theorem~\ref{step4thm} is a standard Wiener process on $[0,\tau]$. 
In particular, for any $n\in\mathbb N$, the random vectors
\begin{equation}\label{rv Z}
Z_i:=\sqrt{\frac n\tau}\left(W(i\tau/n)-W((i-1)\tau/n)\right),\quad i\in[n],
\end{equation}
compose an i.i.d.\ sample of size $n$ from the $d$-dimensional standard normal distribution.

Suppose that one replaces $Z_i$ by $\hat Z^{(T)}_i$ based on $\hat W^{(T)}:=\mathscr T_{\hat\theta_T}(\hat\eta^{(T)})$. 
That is, one considers
\begin{equation}\label{rv hatZ}
\hat Z^{(T)}_i:=\sqrt{\frac n\tau}\left(\hat W^{(T)}(i\tau/n)-\hat W^{(T)}((i-1)\tau/n)\right),\quad i\in[n].
\end{equation}
It will now be convenient to allow $n$ to depend on $T$, so we will henceforth use the notation $n(T)$. 
By the preceding arguments, it follows that $(\hat Z^{(T)}_i)_{i\in[n(T)]}$ behaves asymptotically as a sample of $d$-dimensional standard normal random vectors, as $T\to\infty$ and $n(T) \to \infty$, in such a way that $T/n(T)\to\infty$. 
Hence, replacing the parametric null \eqref{H_0 eq} by
\begin{equation}\label{H_0 prime}
H_0^{(T)}: (\hat Z^{(T)}_1,\ldots,\hat Z^{(T)}_{n(T)}) \text{ is an i.i.d.\ sample from the }\mathcal N(0,I_d)\text{-distribution},
\end{equation}
where $I_d$ denotes the $d\times d$ identity matrix, leads to asymptotically correct tests, i.e., tests having the correct rejection rates under $H_0$, asymptotically as $T\to\infty$. 
The hypothesis~\eqref{H_0 prime} can be evaluated using standard normality tests; well-known examples include (multivariate versions of) the Kolmogorov-Smirnov, Cramér-von Mises and Anderson-Darling tests.

In order to obtain high power against alternatives, we want to consider a large sample size $n=n(T)$. 
However, for fixed $T$, a lower value of the hyperparameter $n$ amounts to more data being used to calculate $\hat Z_i^{(T)}$, hence better approximations to standard normals, yielding a more robust testing procedure. 
In other words, there is a trade-off between small and large $n$. 
One way to find a middle ground is by taking $n=\mathrm{ceil}(c\sqrt T)$ for some $c>0$, where $\mathrm{ceil}(\cdot)$ denotes the ceiling function. 
A smaller choice of $c$ prioritizes the robustness of the testing procedure, whereas a larger $c$ prioritizes power.
Given a sample $\omega_T$ of a point process $N$ on $[0,T]$, one may then test \eqref{H_0 eq} using Algorithm~\ref{alg1} below.

In the present situation, one could also compose a ``naive'' testing procedure ignoring estimation uncertainty, by simply treating the (standardized) empirical process 
\begin{equation}\label{tilde W}
\tilde W^{(T)}:=\mathrm{diag}\left(\sqrt{\mu_{\theta_0}^{(1)}},\ldots,\sqrt{\mu_{\theta_0}^{(d)}}\right)^{-1}\hat\eta^{(T)}
\end{equation}
as a standard Wiener process, hence by testing the random vectors
\begin{equation}\label{rv tildeZ}
\tilde Z^{(T)}_i:=\sqrt{\frac n\tau}\left(\tilde W^{(T)}(i\tau/n)-\tilde W^{(T)}((i-1)\tau/n)\right),\quad i\in[n],
\end{equation}
for normality. 
This amounts to leaving out step \emph{(iii)} of Algorithm \ref{alg1}. 
In Section~\ref{Section5.1}, we will compare this naive testing procedure to our asymptotically correct testing procedure using Monte Carlo simulations, and we show that it is generally a bad idea to ignore the estimation uncertainty induced by the estimation error $|\hat\theta_T-\theta_0|$.

\begin{algorithm}[H]
    \caption{Asymptotically distribution-free goodness-of-fit testing based on the transformed empirical process}\label{alg1}
    \begin{algorithmic}
        \State \textbf{Input:} A realization $\omega_T$ of a multivariate point process $N$ on $[0,T]$ and hyperparameters $\tau\in(0,1)$ and $n=n(T)$, where we use $n=\mathrm{ceil}(c\sqrt T)$ for some $c>0$
        \State \textbf{Output:} Outcome of a goodness-of-fit test
        
        \State \textbf{Algorithmic Steps}
        
        \State \textit{(i)} Estimate $\theta$ using an estimator $\hat\theta_T$; e.g., use maximum likelihood estimation
        
        \State \textit{(ii)} Compute the compensated empirical process $\hat\eta^{(T)}$:
        \begin{equation*}
            \hat\eta^{(T)}: u \mapsto \frac{1}{\sqrt{T}} \left(N(uT) - \int_0^{uT} \lambda_{\hat\theta_T}(s)\, \mathrm{d}s \right), \qquad u \in [0,1]
        \end{equation*}
        
        \State \textit{(iii)} For $\mu_{\hat\theta_T}^{(k)}=\mathbb E[N_{\hat\theta_T}^{(k)}(1)]\approx N^{(k)}(T)/T$, compute the transformed process $\hat W^{(T)}= \mathscr{T}_{\hat\theta_T}(\hat\eta^{(T)})$ by taking, for $u\in[0,\tau]$,
        \begin{equation*}
            \mathscr{T}_{\hat\theta_T}(\hat\eta^{(T)})(u) = \left\{ \frac{1}{\sqrt{\mu_{\hat\theta_T}^{(k)}}} \left((\hat\eta^{(T)})^{(k)}(u) - \int_0^u \frac{(\hat\eta^{(T)})^{(k)}(1) - (\hat\eta^{(T)})^{(k)}(v)}{1 - v}\, \mathrm{d}v \right) \right\}_{k \in [d]}
        \end{equation*}
        
        \State \textit{(iv)} For $i \in [n]$, compute:
        \begin{equation*}
            \hat Z^{(T)}_i = \sqrt{\frac{n}{\tau}} \left(\hat W^{(T)}(i\tau/n) - \hat W^{(T)}((i-1)\tau/n)\right)
        \end{equation*}
        
        \State \textit{(v)} Perform an (asymptotically) exact normality test on the sample $(\hat Z^{(T)}_i)_{i\in[n]}$; e.g., use  the Kolmogorov-Smirnov, Cramér-von Mises or Anderson-Darling test
    \end{algorithmic}
\end{algorithm}

\begin{remark} 
We know from Corollary~\ref{step1+2} that $\tilde W^{(T)}$ converges to a Wiener process plus a random linear drift term. 
The magnitude of the bias of $\tilde Z_i^{(T)}$ caused by this drift is then of order $1/\sqrt n$. 
Since typical test statistics based on a functional of the empirical process --- such as the Kolmogorov-Smirnov, Cramér-von Mises and Anderson-Darling tests --- converge at rate $\sqrt {n(T)}\asymp\sqrt T$, the bias of the naive testing procedure cannot be mitigated through a suitable choice of the hyperparameter $n$.
\end{remark}

\begin{remark}
Taken together, steps \emph{(iii)} and \emph{(iv)} of Algorithm~\ref{alg1} require numerical integration of the compensator over the interval $(0,u)$, where $u = i\tau/n$ for $i \in [n]$. 
Performing this integration requires evaluating the compensator on a grid contained in $(0,u)$. 
For computational efficiency, we recommend choosing a design where each such grid is a subset of a fixed grid on $(0,\tau)$. 
The compensator can then be evaluated \emph{once} on this global grid and stored in advance, after which the required numerical integrations can be carried out efficiently by restriction to the relevant subintervals.
\end{remark}

\subsection{Random-time-change-based goodness-of-fit testing}\label{sectionRTC}
Frequently, when goodness-of-fit tests are conducted in the context of univariate point processes, one considers transformed event times $\Lambda_{\hat\theta_T}(t_i)$, where $\Lambda_\theta(t)=\int_0^t\lambda_\theta(t) \, \mathrm dt$ is the compensator of the point process, i.e., the time-integrated conditional intensity. 
By the random time change theorem \citep[Theorem 7.4.I]{DaleyVereJones}, it holds that, under $\mathbb P_{\theta_0}$, the times $\Lambda_{\theta_0}(t_i)$ are distributed according to a unit Poisson process.\footnote{By a multivariate version of the random time change theorem \citep[Theorem 14.6.IV]{DaleyVereJones}, this approach also allows for a multivariate generalization.}  

Hence, following \cite{Ogata2}, a goodness-of-fit test is often performed by testing whether the increments 
$\Lambda_{\hat\theta_T}(t_{i})-\Lambda_{\hat\theta_T}(t_{i-1})$ are standard exponential.\footnote{See e.g., \citet{RTC2, RTC5, RTC7, RTC1, RTC4, RTC9, RTC8, RTC3, RTC6}.}
The estimation error $|\hat\theta_T-\theta_0|$, leading to a difference between $\Lambda_{\hat\theta_T}$ and $\Lambda_{\theta_0}$, is then typically ignored.
Let us now consider the effect of estimation uncertainty in some more detail.
By Theorem~\ref{step2thm} with $u=t_i/T=\mathcal O(1/T)$, holding for fixed $t_i$, it follows that under $\mathbb P_{\theta_0}$, as $T\to\infty$,
\begin{equation}\label{diff transformed interarrival times}
\Lambda_{\hat\theta_T}(t_{i})-\Lambda_{\hat\theta_T}(t_{i-1})= \Lambda_{\theta_0}(t_{i})-\Lambda_{\theta_0}(t_{i-1})-\frac1{\sqrt T}(t_{i}-t_{i-1})\alpha_{\theta_0}Z+o_{\mathbb P_{\theta_0}}(1/\sqrt T).
\end{equation}
By the random time change theorem, this implies that the transformed interarrival times $\Lambda_{\hat\theta_T}(t_{i})-\Lambda_{\hat\theta_T}(t_{i-1})$ approximately equal standard exponential random variables \emph{plus} some bias term proportional to the interarrival times $t_{i}-t_{i-1}$. 
This bias term entails that the goodness-of-fit testing procedure of assessing whether the times $(\Lambda_{\hat\theta_T}(t_i))$ compose a realization of a unit Poisson process is incorrect. 
Indeed, when one uses a test statistic that is a functional of the empirical process --- such as the Kolmogorov-Smirnov, Cramér-von Mises and Anderson-Darling tests --- to analyze the transformed interarrival times, the convergence of the test statistics occurs at rate $\sqrt {n(T)}\asymp\sqrt T$. 
Therefore, the deviations $\frac1{\sqrt T}(t_{i}-t_{i-1})\alpha_{\theta_0}Z+o_{\mathbb P_{\theta_0}}(1/\sqrt T)=\mathcal O_{\mathbb P_{\theta_0}}(1/\sqrt T)$ in \eqref{diff transformed interarrival times} from standard exponentials are in general \emph{not} negligible. 

If a consistent estimator of $\alpha_{\theta_0}Z$ could be constructed from the data, one could modify \eqref{diff transformed interarrival times} to develop an asymptotically distribution-free goodness-of-fit test based on the random time change theorem. 
However, this is not at all a straightforward task, and outside the scope of the present work.
Although tests based on the random time change theorem as discussed in this subsection rely on different ideas than the goodness-of-fit testing procedure outlined in Algorithm~\ref{alg1}, it is, of course, possible to compare the outcomes of both procedures numerically; see Section~\ref{Section5.1} and Appendix~\hyperref[app C]{C}.

\section{Consistency under $H_1$}\label{Section4.2}
In this section, we analyze the consistency under $H_1$ of the goodness-of-fit testing procedure outlined in Algorithm~\ref{alg1}. 
In step \emph{(i)} of Algorithm~\ref{alg1}, we fit $N$ to the parametric hypothesis $H_0$ by calculating the maximum likelihood estimator $\hat\theta_T$. 
For the normality test in step \emph{(v)} of Algorithm~\ref{alg1}, we consider a test that is consistent under alternatives at rate $\sqrt{n}(T)$, such as a Kolmogorov-Smirnov,  Cramér-von Mises or Anderson-Darling test. 

Before stating the main result of this section, we state a lemma describing the asymptotic behavior of $\hat W^{(T)}=\mathscr T_{\hat\theta_T}(\hat\eta^{(T)})$. 
The proofs for this section can be found in Appendix~\hyperref[app B]{B}.

\begin{lemma}\label{thm5lemma}
Grant Assumptions \ref{assA}--\ref{assB} and \ref{MLEass}. 
Consider a realization of a stationary and ergodic point process $N$ having law $\mathbb  P$ and conditional intensity $\lambda\notin\mathscr L_\Theta:=\{\lambda_\theta:\theta\in\Theta\}$, i.e., $N$ does not satisfy the null \eqref{H_0 eq}.
Let $\theta^*\in\Theta$ be the ``least-false estimator'', or ``oracle'', maximizing the expected likelihood $\mathbb E[L_T(\theta)]$ over $\theta\in\Theta$. 
Assume that $\int_0^\infty\|\mathbb E[\lambda(s)-\lambda_{\theta^*}(s)-\mathbb E[\lambda(0)-\lambda_{\theta^*}(0)]|\mathcal H_0]\|_2\ \mathrm ds<\infty$ and $\mathrm{Var}(\lambda(0)-\lambda_{\theta^*}(0))<\infty$. 
Select $\tau\in[0,1)$.
Let $g_k(s):=\lambda^{(k)}(s)-\lambda_{\theta^*}^{(k)}(s)$, $k\in [d]$. 
Then the transformed empirical process $\hat W^{(T)}=\mathscr T_{\hat\theta_T}(\hat\eta^{(T)})$ converges to $W_1+W_2+V_3$, weakly on $D([0,\tau],\mathbb R^d)$ equipped with the strong Skorokhod $J_1$ topology, where $W_1$ is a standard Wiener process, $W_2$ is a Wiener process with covariance 
\begin{equation}
    u\mapsto u\cdot\left(2\int_0^{\infty}\mathbb E\left[\frac{g_i(0)-\mathbb E[g_i(0)]}{\sqrt{\mu_{\theta^*}^{(i)}}}\frac{g_j(t)-\mathbb E[g_j(0)]}{\sqrt{\mu_{\theta^*}^{(j)}}}\right]\ \mathrm dt\right)_{i,j\in[d]},\label{covW1}
\end{equation}
and $V_3(u)=\int_0^uW_3(s)\, \mathrm ds$, with $W_3\stackrel{d}=W_2$. 
\end{lemma}
\begin{remark}
The condition $\int_0^\infty\|\mathbb E[\lambda(s)-\lambda_{\theta^*}(s)-\mathbb E[\lambda(0)-\lambda_{\theta^*}(0)]|\mathcal H_0]\|_2\ \mathrm ds<\infty$ ensures intuitively that the influence of the past (before time $0$) on the future (after time $t$) diminishes rapidly enough (as $t\to\infty$), i.e., that the process does not have a(n excessively) long memory.
\end{remark}

\begin{theorem}\label{consistency thm}
Grant the conditions of Lemma~\ref{thm5lemma}.
Suppose that $\lambda$ does \emph{not} satisfy \begin{equation}
    \lambda\stackrel{d}=\lambda_\theta+c\quad\text{for some}\;\theta\in\Theta\;\text{and some}\; c\in\mathbb R^d,\label{nolinearcomb}
\end{equation} 
and test the null \eqref{H_0 eq} using Algorithm \ref{alg1}, employing a normality test that is consistent under alternatives at rate $\sqrt{n(T)}$ in step \emph{(v)}. 
Then this goodness-of-fit test has non-trivial power, asymptotically. 
Furthermore, for $W_1$ and $W_2$ as in Lemma~\ref{thm5lemma}, suppose that 
\begin{equation}
    \label{weirdcondition}
    \text{there exist }s<t, s,t\in[0,\tau]:\quad W_1(t)+W_2(t)-W_1(s)-W_2(s)\not\sim\mathcal N(0,t-s).
\end{equation}
Then  
$    \mathbb P\left(\text{Algorithm \ref{alg1} rejects }H_0 \right)\to1,\ \text{as }n(T)\to\infty.$
\end{theorem}

\begin{remark}
    Condition~\eqref{weirdcondition} (slightly) constrains the covariance structure between $W_1$ and $W_2$. 
    Since $W_1$ is a standard Wiener process and $W_2$ is a non-zero Wiener process, this condition is rather mild.\footnote{In general, for Wiener processes $W_1$ and $W_2$, with $W_1$ standard Wiener, that are dependent in a \emph{pathological} way, it is possible that $W_1 + W_2$ reduces to a standard Wiener process. 
    For example, if $W$ is a standard Wiener process independent of $W_1$, and $W_2=2\rho\left(\sqrt{1-\rho^2}W-\rho W_1\right)$, $\rho\in(0,1)$, then $W_1+W_2$ reduces to a standard Wiener process.} 
    
\end{remark}

A natural follow-up question is whether the conclusions of Theorem~\ref{consistency thm} continue to hold under a (more challenging) \emph{root-$n(T)$} alternative hypothesis. 
Since we use $n(T)\asymp\sqrt T$, this is the same as a ``$T^{1/4}$ alternative''. 
Specifically, consider a family of alternative models $\breve{N}_T$ on $[0, T]$ with conditional intensity functions of the form 
\begin{equation}
    \breve{\lambda}_T = \lambda_\theta + \breve{\lambda}/T^{1/4},\label{4rootT def}
\end{equation}
where $\theta \in \Theta$ and $\breve{\lambda}$ is the conditional intensity of some other point process on $[0,T]$.

Under such alternatives, the perturbation introduced by $V_3$ is of order $1/n(T)$ and is therefore asymptotically undetectable. 
Consequently, the focus should be on the contribution of $W_2$. As in Theorem~\ref{consistency thm}, suppose that condition~(\ref{weirdcondition}) holds. 
In that case, $W_2$ induces a perturbation of order $1/\sqrt{n(T)}$, which as we establish below leads to non-trivial power in the goodness-of-fit test. 
As may be expected, such root-$n(T)$ alternatives do not guarantee that the test's power converges to one.

\begin{corollary}\label{rootT}
Consider the setup of Theorem~\ref{consistency thm}, but with root-$n(T)$ alternatives of the form \eqref{4rootT def}. 
Test the null \eqref{H_0 eq} using Algorithm~\ref{alg1}, employing a normality test that is consistent under alternatives at rate $\sqrt{n(T)}$ in step \emph{(v)}.
Then this test has non-trivial power, asymptotically. 
\end{corollary}

The proof of Theorem~\ref{consistency thm} implies that if the test from Section~\ref{Section4.1} is not consistent under an ergodic alternative with intensity $\lambda$, then it holds for some $c\in\mathbb R^d$ that $\lambda(t)-\lambda_{\theta^*}(t)\to c$, as $t\to\infty$. 
Under a stationary alternative, this is only possible if $\lambda\stackrel{d}=\lambda_{\theta^*}+c$ to begin with. 
By choosing a model ``closed under adding constant intensities'' (or: a model with a parametric intensity specification including a linear parametric term), i.e., a model such that
$$
\{\lambda_\theta+c : \theta\in\Theta, c\in\mathbb R^d, \lambda_\theta(s)+c\geq0\text{ for all }s\geq0\text{ a.s.}\} \subseteq \mathscr L_\Theta,
$$
we obtain consistency under any \emph{stationary} deviation from $H_0$. 
Note that this implies that it is a good idea to add a linear parametric term to the conditional intensity, if one wishes to ensure consistency under $H_1$. 
If our model is not closed under adding constant intensities, then the only stationary deviations from the null that our test might fail to detect are point processes having an intensity that is a member of
$$
\{\lambda_\theta+c : \theta\in\Theta, c\in\mathbb R^d\setminus\{0\}, \lambda_\theta(s)+c\geq0\text{ for all }s\geq0\text{ a.s.}\}.
$$
On the other hand, even for a null model that is closed under adding constant intensities, it is still possible that our test fails to detect an ergodic, but nonstationary deviation $N$ from the null, but where $\lambda\to\lambda_\theta$ a.s., for some $\theta\in\Theta$. 
However, one might argue that any such model is in fact ``close'' to the null, since it converges to a null model as it converges to stationarity.

\section{Monte Carlo simulations}\label{Section5.1}
In this section, we investigate the finite-sample performance of the goodness-of-fit procedures described in Section~\ref{SectionGoF}. 
We analyze both their size and power properties.
We compare our asymptotically exact testing procedure outlined in Algorithm~\ref{alg1} to the two testing procedures in Section~\ref{SectionGoF} that ignore estimation uncertainty --- naive and random-time-change-based ---, and 
show that our testing procedure has substantial advantages over those tests. 

We consider parametric null hypotheses consisting of self-exciting Hawkes processes, the parameters of which we estimate from simulated sample paths of a variety of temporal point processes. 
The Hawkes process was introduced in 1971, see \citet{Hawkes, Hawkes2}, and see, e.g., \citet{Bacry,ALP14,ACL15,HawkesSurvey22,Ikefuji} and the references therein for more recent advances. 
More specifically, for the parameter space $\Theta=(0,10)^3$, 
we consider the parametric null hypotheses
\begin{align}\label{H0 exp}
H_0^{\mathrm{Haw:E}}:\  &N\sim\mathbb P_\theta^{\mathrm{Haw:E}}\text{ for some }\theta\in\{(\mu,\alpha,\beta)\in\Theta:\alpha<\beta\},\\
\label{H0 PL}
H_0^{\mathrm{Haw:P}}:\  &N\sim\mathbb P_\theta^{\mathrm{Haw:P}}\text{ for some }\theta\in\{(\mu,\alpha,\beta)\in\Theta:\alpha<\beta\},
\end{align}
where $\mathbb P_\theta^{\mathrm{Haw:E}}=\mathbb P_{\mu,\alpha,\beta}^{\mathrm{Haw:E}}$ (resp.\ $\mathbb P_\theta^{\mathrm{Haw:P}}=\mathbb P_{\mu,\alpha,\beta}^{\mathrm{Haw:P}}$) is the law of a 
linear exponential (resp.\ power-law) Hawkes process with conditional intensity process given in Table~\ref{lambdatabel}.

For both hypotheses, we simulate realizations on $[0,T]$ for $T=5{,}000$ and $T=50{,}000$.  
We set $n=\mathrm{ceil}(\sqrt T/c)$ for both our asymptotically correct, transformation-based testing procedure and the naive testing procedure based on the untransformed empirical process \eqref{tilde W}. 
The choice of the hyperparameter $n$ is motivated in Section~\ref{Section4.1}. 
It turns out that $c=1/4$ yields accurate rejection rates for various specifications of $\mathbb P_{\mu,\alpha,\beta}^{\mathrm{Haw:\cdot}}$.
The hyperparameter $\tau$ from Algorithm~\ref{alg1} is set to $0.9$.

\begin{table}[t]
\centering
\footnotesize
\caption{\footnotesize Conditional intensities of the simulated temporal point processes}

\begin{tabular}{l | l | l}\label{lambdatabel}

Name & Point process & Form of the conditional intensity  \\
\hline
Exponential Hawkes & $N\sim\mathbb P^{\mathrm{Haw:E}}:=\mathbb P_{1/2,1,2}^{\mathrm{Haw:E}}$ & $\displaystyle\lambda_{\mu,\alpha,\beta}^{\mathrm{Haw:E}}(t):=\mu + \sum_{t_i<t} \alpha e^{-\beta(t - t_i)}$  \\
Power-law Hawkes & $N\sim\mathbb P^{\mathrm{Haw:P}}:=\mathbb P_{1/2,1,2}^{\mathrm{Haw:P}}$ & $\displaystyle\lambda_{\mu,\alpha,\beta}^{\mathrm{Haw:P}}(t):=\mu + \sum_{t_i<t} \alpha (1 + t - t_i)^{-\beta}$\\
Shot-noise & $N\sim \mathbb P^{\mathrm{SN}}:=\mathbb P^{\mathrm{SN}}_{1/5,10,2}$ &  $\displaystyle\lambda_{\mu,\alpha,\beta}^{\mathrm{SN}}(t):=\sum_{t_i<t\  \mathrm{ Poi}(\mu)\text{ times}} \alpha e^{-\beta (t - t_i)}$ \\
Inhomogeneous Poisson & $N\sim \mathbb P^{\mathrm{Periodic}}:=\mathbb P^{\mathrm{Periodic}}_{5/4,1,1/5,0}$ &  $\lambda_{\mu,\alpha,\beta}^{\mathrm{Periodic}}(t):=\mu + \alpha \sin(\beta(t - \gamma))$ \\
Self-correcting & $ N\sim \mathbb P^{\mathrm{SC}}:=\mathbb P^{\mathrm{SC}}_{1,1/2,\log(2)}$ &  $\lambda_{\mu,\alpha,\beta}^{\mathrm{SC}}(t):=\mu e^{\beta t} \alpha^{N(t_-)}$ 
\end{tabular}
\end{table}

\begin{table}[t]\centering
\caption{{\footnotesize Using the parametric null hypothesis $H_0^{\mathrm{Haw:E}}$, $T=\{5{,}000; 50{,}000\}$ and $n=\mathrm{ceil}(\sqrt T/4)$, we conduct $500$ simulations for one model contained in $H_0$ and four models not contained in $H_0$. 
We report the number of rejections, $R_{0.01}; R_{0.05}; R_{0.20}$, out of $500$ tests, at significance levels $0.01; 0.05; 0.20$, using the transformation-based testing procedure outlined in Algorithm \ref{alg1}, employing Kolmogorov-Smirnov (KS), Cramér-von Mises (CvM) and Anderson-Darling (AD) tests in step \emph{(v)} thereof, and the naive testing procedure that leaves out step \emph{(iii)}.
Panel~(a) considers $T=5{,}000$ and panel~(b) $T=50{,}000$.}}

\bitsmall
\begin{tabular}{c c | c c c | c c c}
	&&\multicolumn{3}{c}{Transformation-based testing procedure}&\multicolumn{3}{c}{``Naive'' testing procedure}\\
& Test&KS&CvM&AD&KS&CvM&AD\\
\hline
& $\mathbb P^{\mathrm{Haw:E}}$         &8; 39; 108  &9; 34; 116  &8; 32; 123  &0; 1; 6    &0; 0; 7    &0; 1; 12  \\
& $\mathbb P^{\mathrm{Haw:P}}$          &7; 35; 109  &5; 32; 97   &9; 35; 113  &0; 0; 7    &0; 0; 3    &0; 1; 14  \\
(a) & $\mathbb P^{\mathrm{SN}}$ &7; 44; 237  &0; 33; 256  &0; 27; 262  &0; 5; 96   &0; 7; 154  &0; 10; 170\\
& $\mathbb P^{\mathrm{Periodic}}$     &43; 242; 480&16; 322; 497&12; 341; 499&3; 170; 454&7; 274; 493&2; 298; 499\\
& $\mathbb P^{\mathrm{SC}}$           &500; 500; 500&500; 500; 500&500; 500; 500&500; 500; 500&500; 500; 500&500; 500; 500\\

\hline
& $\mathbb P^{\mathrm{Haw:E}}$			&6; 25; 110	&5; 20; 107	&4; 25; 108	&0; 0; 9	&0; 0; 5	&0; 0; 6	\\
& $\mathbb P^{\mathrm{Haw:P}}$			&9; 31; 101	&8; 30; 113	&9; 37; 144	&0; 0; 11	&0; 0; 9	&0; 2; 28	\\
(b) & $\mathbb P^{\mathrm{SN}}$ &83; 258; 464	&68; 337; 488	&143; 418; 499	&24; 166; 441	&43; 278; 478	&104; 396; 494	\\
& $\mathbb P^{\mathrm{Periodic}}$		&495; 500; 500	&499; 500; 500	&500; 500; 500	&494; 500; 500	&500; 500; 500	&500; 500; 500	\\
& $\mathbb P^{\mathrm{SC}}$			&500; 500; 500	&500; 500; 500	&500; 500; 500	&500; 500; 500	&500; 500; 500	&500; 500; 500	\\

\end{tabular}\normalsize\label{table1}
\normalsize
\end{table}

Using the transformation-based and naive testing procedures, we perform goodness-of-fit tests for $500$ simulated realizations of the various processes listed in Table~\ref{lambdatabel}.
We first consider simulations under $H_0^{\mathrm{Haw:E}}$. 
We present the results for the ``small'' time horizon $T=5{,}000$ in Table~\ref{table1}, panel~(a), and for the ``large'' time horizon $T=50{,}000$ in Table~\ref{table1}, panel~(b),  
employing Kolmogorov-Smirnov, Cramér-von Mises and Anderson-Darling test statistics.

First, we find that for model $\mathbb P^{\mathrm{Haw:E}}$ (belonging to $H_0^{\mathrm{Haw:E}}$), the rejection rates when using our transformation-based testing procedure align well with the nominal significance levels. 
This confirms the asymptotic correctness of the transformation-based procedure, as stated in Section~\ref{Section4.1}. 
By contrast, the naive testing procedure, based on the untransformed empirical process and ignoring estimation uncertainty, produces \emph{undersized} tests: 
for $\mathbb P^{\mathrm{Haw:E}}$, the numbers of rejections at significance levels of $0.01$, $0.05$, and $0.20$ are far below the expected counts of 5, 25, and 100, respectively, both when $T=5{,}000$ and when $T=50{,}000$. 
This demonstrates that ignoring estimation uncertainty leads to unreliable tests.
Furthermore, the power of the transformation-based procedure against alternative hypotheses appears to be considerably higher, in all instances, than that of the naive procedure. 
We also find that, whereas our procedure has high power against the alternatives $\mathbb P^{\mathrm{SN}}$, $\mathbb P^{\mathrm{Periodic}}$ and $\mathbb P^{\mathrm{SC}}$ that further increases with the sample period, the alternative $\mathbb P^{\mathrm{Haw:P}}$ --- which is itself a self-exciting Hawkes process just like the null model $\mathbb P^{\mathrm{Haw:E}}$, with parameter values that have deliberately been selected to mimic to null model --- is more challenging to distinguish from the null; we analyze this case in further detail in Appendix~\hyperref[app C]{C}.\footnote{A second-order Taylor series expansion of the power-law Hawkes kernel is, under the current parameter values, nearly identical to that of the exponential Hawkes kernel. 
Their difference can therefore only be statistically detected in extensive samples.}

\begin{figure}[H]\centering

\begin{tabular}{cc}
\subfloat[\scriptsize Transformation-based testing procedure and $T=5{,}000$]{\includegraphics[width = 3.2in]{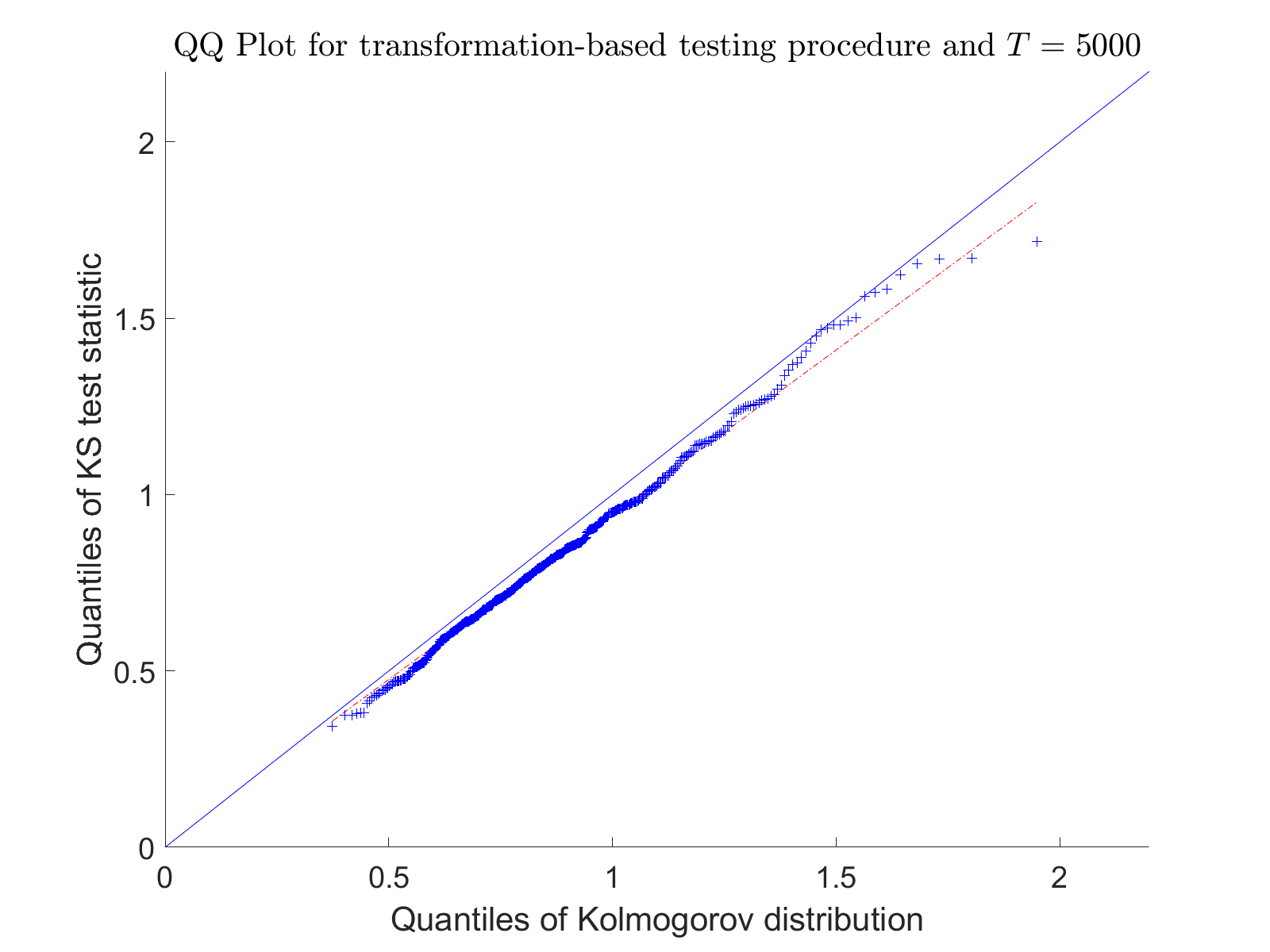}} &
\subfloat[\scriptsize``Naive'' testing procedure and $T=5{,}000$]{\includegraphics[width = 3.2in]{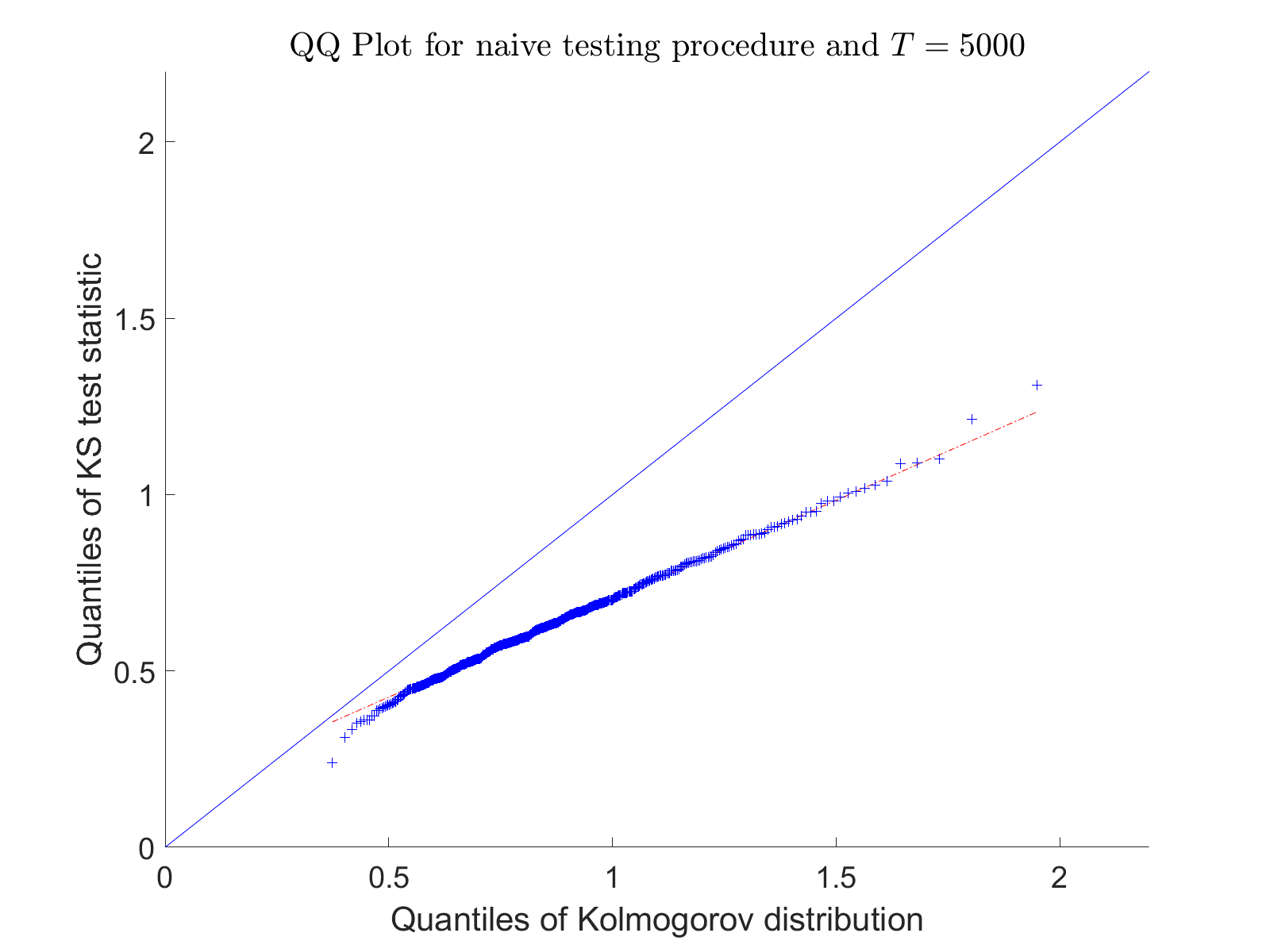}} \\
\subfloat[\scriptsize Transformation-based testing procedure and $T=50{,}000$]{\includegraphics[width = 3.2in]{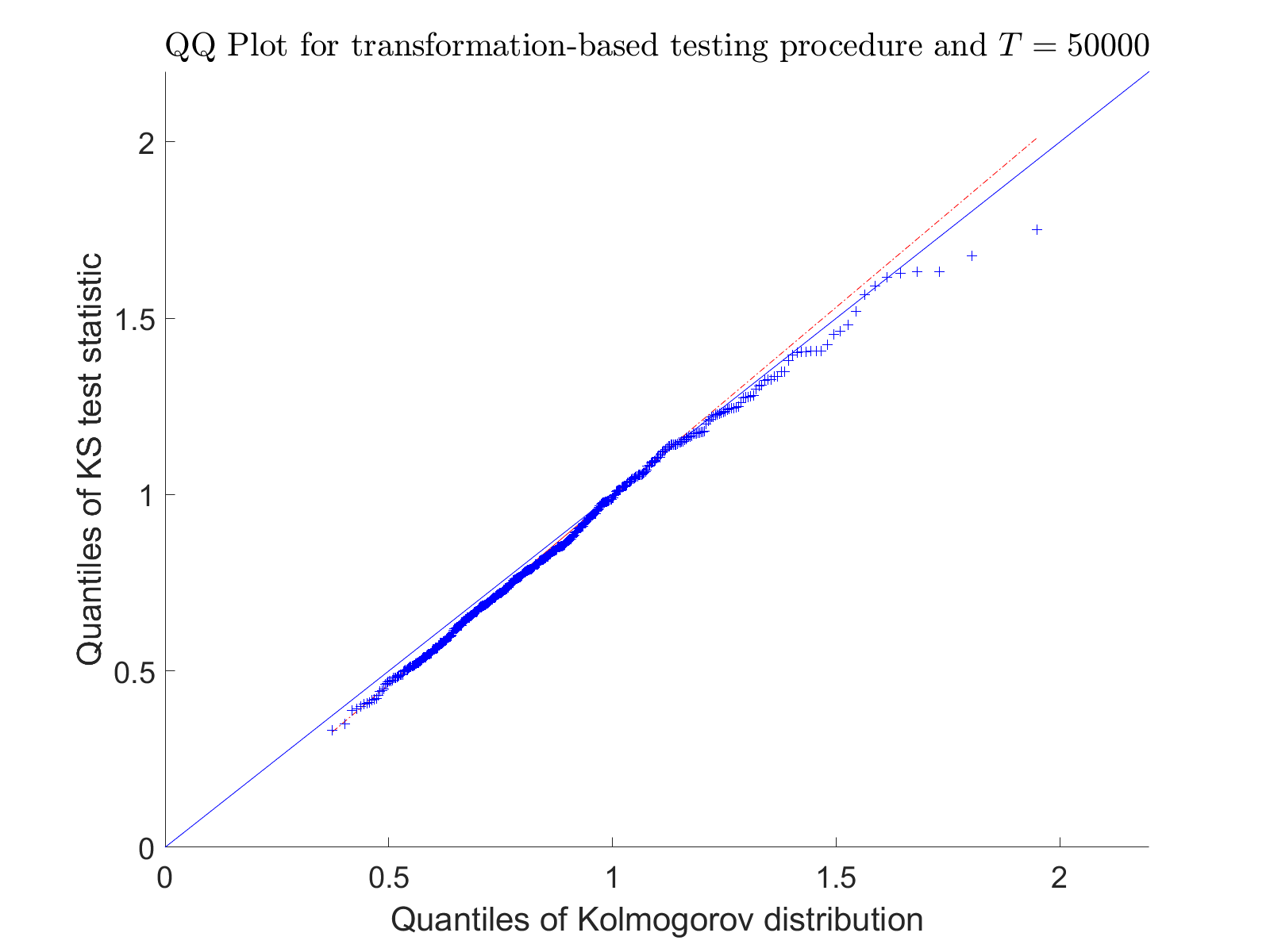}} &
\subfloat[\scriptsize``Naive'' testing procedure and $T=50{,}000$]{\includegraphics[width = 3.2in]{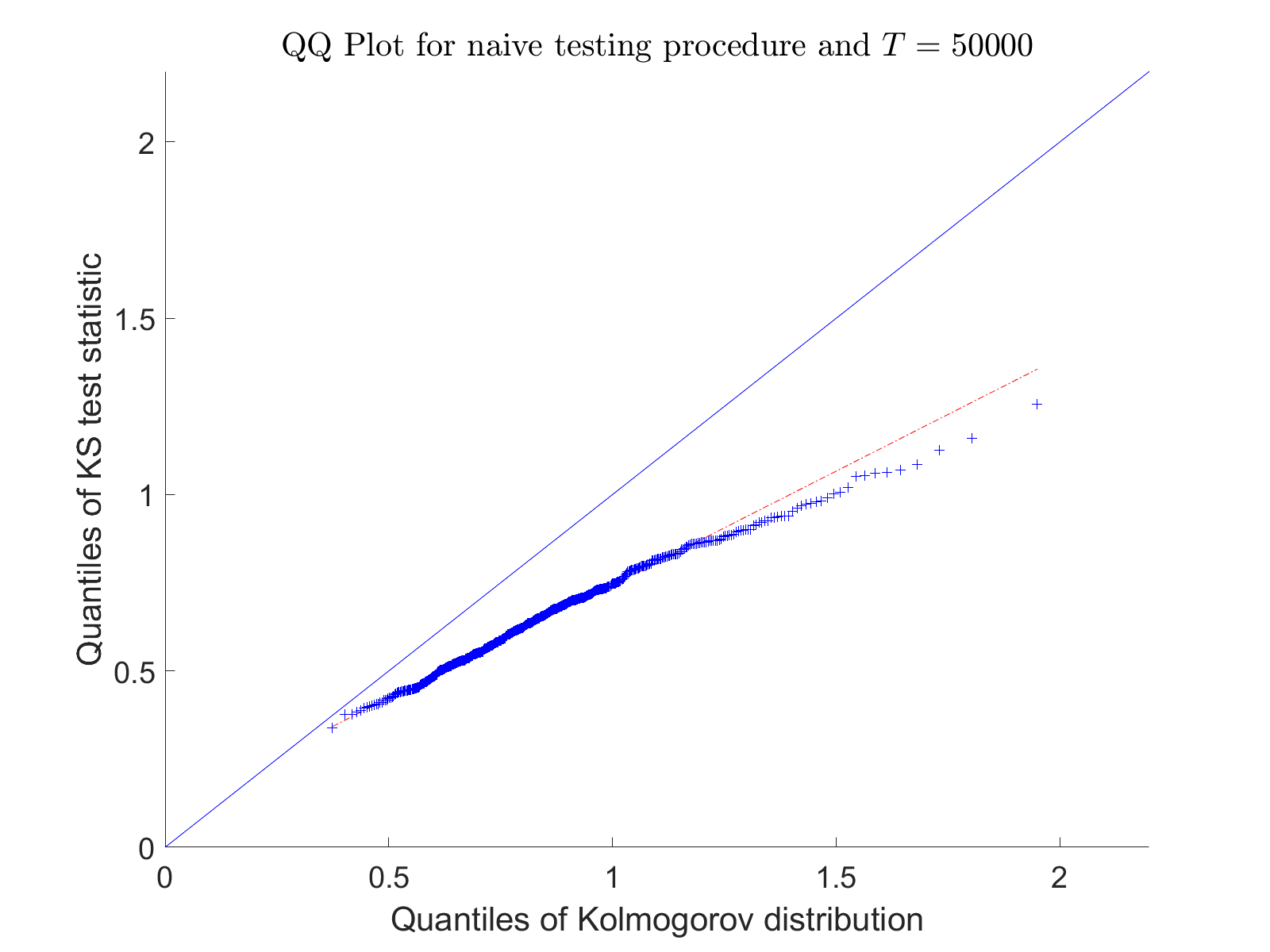}} 
\end{tabular}
\caption{\label{fig1} {\bitsmall We consider $500$ simulations of $N\sim \mathbb P^{\mathrm{Haw:E}}$, both for $T=5{,}000$ and $T=50{,}000$, and both for our transformation-based and the ``naive'' testing procedures outlined in Section \ref{Section4.1}. 
We test for standard normal increments using the Kolmogorov-Smirnov test. 
The figure displays Q-Q plots for the empirical distribution of the test statistics vs.\ the theoretical limit distribution, i.e., the \emph{Kolmogorov distribution}. 
The dashed red line is a straight line fitted to the scatter plot, whereas the line $y=x$ is drawn in solid blue.}}
\normalsize
\end{figure}

Next, we consider in more detail the results of step \emph{(v)} of the testing procedure described in Algorithm~\ref{alg1}. 
Suppose we apply the Kolmogorov-Smirnov normality test in this step.\footnote{Results for the other test statistics are similar.} 
We can then compare the simulated Kolmogorov-Smirnov test statistics for $\mathbb P^{\mathrm{Haw:E}} \in H_0^{\mathrm{Haw:E}}$ with their theoretical null distribution given by the Kolmogorov distribution, using a Q-Q plot. 
Under the null, the Q-Q plot should be centered around the line $y = x$. 
We investigate this for $H_0^{\mathrm{Haw:E}}$ using both $T = 5{,}000$ and $T = 50{,}000$, contrasting the transformation-based and naive testing procedures; see Figure~\ref{fig1}. 
As expected from Sections~\ref{Section3} and~\ref{Section4.1}, the Q-Q plots for the transformation-based procedure align closely with the line $y = x$; see Figures~\ref{fig1}(a) and (c). 
Moreover, the deviation from $y = x$ is smaller when $T = 50{,}000$ compared to when $T = 5{,}000$, which conforms with the asymptotic nature of our testing procedure (as $T \to \infty$). 
By contrast, the naive testing procedure yields different results. 
Since the standardized empirical process in \eqref{tilde W} corresponds to a Brownian motion with a linear drift and a random coefficient, the random variables $\tilde Z_i^{(T)}$ defined in (\ref{rv tildeZ}) are normally distributed, but not standard normal. 
This is reflected in the Q-Q plots for the naive procedure, where the points form an approximately straight line but deviate from $y=x$; see Figures~\ref{fig1}(b) and~(d).

In Appendix~\hyperref[app C]{C}, we provide additional Monte Carlo simulation results, including 
the rejection rates obtained when considering the power-law Hawkes null hypothesis~\eqref{H0 PL} and the rejection rates resulting from the random-time change-based test outlined in Section~\ref{sectionRTC}, for each of the five models listed in Table~\ref{lambdatabel}; 
we also provide further analysis explaining the differences in difficulty of detecting various deviations from the null. 
The results are similar to the ones found in Table~\ref{table1}, and confirm the substantially increased reliability and power of our asymptotically exact testing procedure compared to methods ignoring estimation uncertainty.

\section{Data analyses}\label{Section5.2}
In this section, we illustrate the testing procedures of Section~\ref{SectionGoF} using two real-world data sets previously analyzed in part by \citet{Ogata2} and \citet{Schoenberg}. 
In both cases, our asymptotically exact testing procedure signals lower confidence in the model's goodness-of-fit than the naive and random-time-change-based tests. 
This demonstrates the practical importance of properly accounting for estimation uncertainty, as neglecting it can lead to misleading conclusions --- consistent with the simulation results in Section~\ref{Section5.1} and Appendix~\hyperref[app C]{C}. 
The main text highlights the key findings; various additional details of the real data analyses are provided in Appendix~\hyperref[app D]{D}.

\subsection{Temporal ETAS model}\label{OgataDatasection}
We first revisit the earthquake data analyzed in \citet{Ogata2}, comprising 483 earthquakes. 
In \citet{Ogata2}, these earthquake occurrences are modeled via a temporal marked Hawkes process --- or Epidemic Type Aftershock Sequence (ETAS) model --- with conditional intensity given by
\begin{equation}\label{HawkesOmori}
\lambda_\theta(t) = \lambda_{(\mu, K, c, \beta)}(t) =  \mu + \sum_{t_i < t} e^{\beta(m_i - M)} \frac{K}{t - t_i + c},
\end{equation}
where $m_{i}$ is observable, $M \equiv 6.0$, and $\theta=(\mu, K, c, \beta)$ is estimated via maximum likelihood, yielding $(\hat\mu, \hat K, \hat c, \hat\beta) = (.00536, .017284, .01959, 1.61385)$; see \citet[Table 3]{Ogata2}.
To evaluate the goodness-of-fit, \citet{Ogata2} applies the random-time-change-based goodness-of-fit test outlined in Section~\ref{sectionRTC}. 
A Kolmogorov-Smirnov test for standard exponentiality of the interarrival times $\Lambda_{\hat\theta}(t_{i+1})-\Lambda_{\hat\theta}(t_i)$ yields a $p$-value of $.6380$, suggesting a good fit.

While this classical approach neglects estimation uncertainty, our transformation-based methodology detailed in Algorithm~\ref{alg1} explicitly accounts for the variability introduced by using estimated parameters, yielding an empirical process that approximates a standard Brownian motion under the null.
Figure~\ref{fig2} displays the path of our empirical process. 
Although broadly consistent with the path of a Brownian motion, the presence of abrupt jumps signals deviation from the continuous trajectory of a Brownian motion. 

\begin{figure}[t]
\centering
  \includegraphics[width=.85\linewidth]{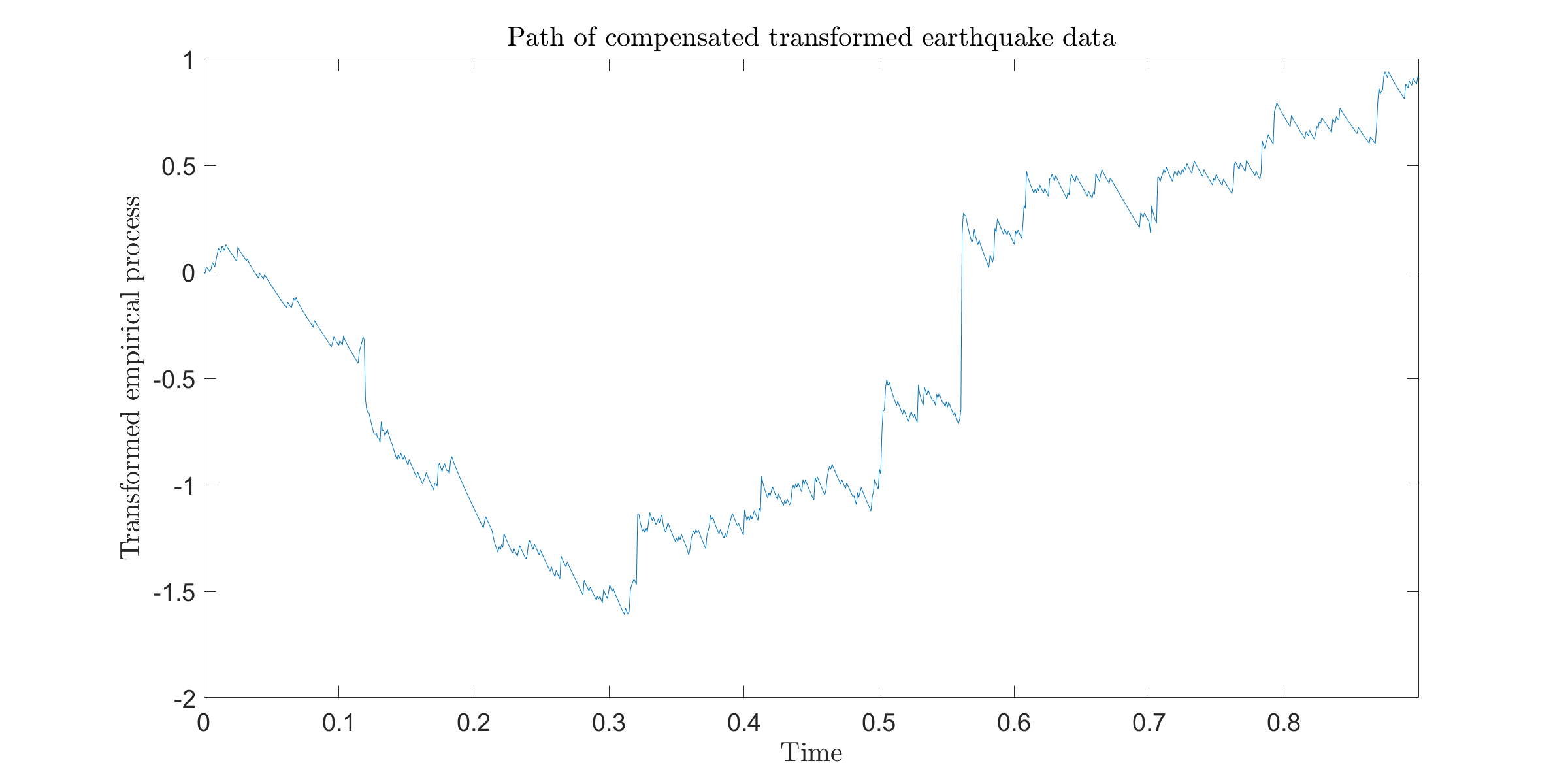}
  \caption{\footnotesize This figure displays the transformed empirical process under the null hypothesis that the conditional intensity satisfies \eqref{HawkesOmori}, using the earthquake data from \cite{Ogata2}. 
  The path should resemble a Brownian motion if the model is well-specified. 
  The presence of jumps indicates deviation from Brownian behavior, reflecting potential model misspecification.}
  \label{fig2}
\end{figure}

Applying our testing procedure with $n = \mathrm{ceil}(\sqrt{483}/4) = 6$ and using the Kolmogorov-Smirnov (KS) test statistic, we obtain a $p$-value of $.1076$; with Anderson-Darling, the $p$-value is $.0568$. 
This is notably lower than the $p$-value of $0.6380$ (KS) from the random-time-change-based test and than the $p$-value of $0.7684$ (KS) that would be obtained from the naive testing procedure of Section~\ref{Section4.1}, both of which do not correct for estimation error. 
Although the null is not rejected at a typical significance level of $0.05$, both the near-significance and the qualitative behavior of the path in Figure~\ref{fig2} raise concerns about the adequacy of the temporal ETAS model. 
These findings support the broader claim that ignoring estimation uncertainty can lead to misleading conclusions about model fit, underscoring the value of our transformation-based approach.

\subsection{A recursive point process model}\label{Schoenbergdatasection}

\citet{Schoenberg} introduces a novel recursive point process model for epidemic data in which the expected number of offspring from an event decreases with the conditional intensity at the event time, modeling saturation effects. 
The model captures the phenomenon that the degree of transmission is higher when a disease is rare within a population, such as during the early stages of an outbreak; at this stage, infected individuals are more likely to spread the disease to susceptible individuals. 
By contrast, as the disease becomes more prevalent, the transmission rate declines due to the increasing number of individuals who have acquired immunity or prior exposure and the adoption of preventive measures.
The conditional intensity of the model takes the parametric form  
$$
\lambda_\theta(t)=\lambda_{(\mu,\kappa,\beta,\alpha)}(t)=\mu+\sum_{t_i<t}\kappa/(\lambda(t_i))^{\alpha}e^{-\beta(t-t_i)}.
$$

We first repeat the analysis of \citet{Schoenberg} using Rocky Mountain Spotted Fever (RMSF) data from California ($67$ events), also studied in \citet{Schoenberg}, and next estimate the same recursive point process model to RMSF data from Florida ($191$ events), which were not considered in their original study. 
To assess the goodness-of-fit, we then apply the testing procedures outlined in Section~\ref{SectionGoF}. 


\begin{table}[t]\centering
\caption{{\small $p$-values from three goodness-of-fit testing procedures (see Section~\ref{SectionGoF}) using the Kolmogorov-Smirnov (KS), Cramér-von Mises (CvM), and Anderson-Darling (AD) test statistics. 
An asterisk denotes rejection at the $\alpha=0.05$ level.}}
\label{pvaluestab}
\bitsmall
\begin{tabular}{c | c c c | c c c}
&\multicolumn{3}{c}{California data}&\multicolumn{3}{c}{Florida data}\\
Test&KS&CvM&AD&KS&CvM&AD\\
\hline
Asymptotically exact testing procedure &.5371  &.4599  &.5193  &.0479$^\ast$  &.0116$^\ast$  &.0185$^\ast$\\
Naive testing procedure      &.8435  &.8218  &.8833  &.0965      &.1649      &.1979\\
Random-time-change-based test        &.3104  &.3319  &.1895  &.4836      &.2770      &.0759
\end{tabular}
\end{table}

Table~\ref{pvaluestab} reports $p$-values for the three testing approaches: our proposed asymptotically exact procedure, the naive testing procedure, and the classical test based on the random time change theorem.
For the data from California, we reach the same conclusions as \citet{Schoenberg}: none of the nine tests displayed in the left panel of Table~\ref{pvaluestab} rejects the model.

For the (substantially larger) Florida data set (not considered in \citealp{Schoenberg}), however, only our testing procedure given in Algorithm~\ref{alg1} rejects the recursive model at the $\alpha = 0.05$ level, using any of the three normality tests. 
Both the naive and the random-time-change-based approaches fail to detect the apparent misfit, irrespective of the choice of the normality test. 

We have seen in Section~\ref{Section5.1} that tests that ignore estimation uncertainty are typically undersized and have limited power against alternatives. 
The recursive model appears to fit the California data well. 
The seemingly adequate fit suggested by the naive and random-time-change-based-tests for the Florida data set, containing a considerably larger number of events, is, however, refuted by our asymptotically exact procedure. 
The contrasting outcomes for the Florida data set underscores once more the importance of properly accounting for estimation uncertainty in goodness-of-fit testing.

\section{Concluding remarks}\label{SectionConclusion}
In this work, we have constructed a novel test process for goodness-of-fit testing of parametric temporal point processes and have established its limiting behavior. 
The test process can be used to generate asymptotically distribution-free goodness-of-fit tests. 
The resulting tests are readily implemented and address the limitations of existing tests that ignore estimation uncertainty, which are often undersized and lack power.

Several extensions and follow-up questions can be considered.
First, although our focus has primarily been on \(d\)-dimensional point processes for \(d \in \mathbb{N}\), it is worth noting that our results extend naturally to spatiotemporal point processes defined on an infinite spatial domain \( \mathscr{X} \), without introducing significant challenges. 
In this case, one would work with a conditional intensity function $\lambda_\theta(s,t)$ and analyze space-integrated quantities, assuming the regularity conditions required for the central limit theorem $\sqrt{T}(\theta_0 - \hat{\theta}_T) \stackrel{d}{\to} Z$ are satisfied; see \citet{MLEinfinitedimensional}. 
Additionally, if $\alpha_\theta := \lim_{T \to \infty}\frac1T \int_0^T \int_{\mathscr{X}} \frac{\partial \lambda_\theta(s,t)}{\partial \theta} \, \mathrm{d}s \, \mathrm{d}t$ exists a.s.\ for $\theta = \theta_0$, one can derive a version of Corollary~\ref{step1+2}. 
This results in a Wiener process with variance $\mathbb{E}_{\theta_0}[N([0,1], \mathscr{X})]$ and drift $\alpha_{\theta_0}Z$, which is in a form suitable for applying an innovation martingale transformation analogous to the one considered in this paper. 
Therefore, one can proceed as in Section~\ref{sec3.3} and beyond.  
This framework then allows for the construction of asymptotically distribution-free goodness-of-fit tests for spatiotemporal point processes.


Second, in this work, we employ the increments of the test process constructed in step \emph{(iii)} of Algorithm~\ref{alg1}. 
The central theoretical justification for Algorithm~\ref{alg1} lies in the weak convergence of the test process to a standard Wiener process, as $T \to \infty$. 
This convergence permits a broad class of asymptotically valid test procedures: step \emph{(iv)} may be reformulated as the application of an arbitrary continuous functional to the process in step \emph{(iii)}, while step \emph{(v)} entails evaluating the distribution of this functional under the null via its counterpart applied to a standard Wiener process. 
Our framework thus facilitates the construction of a \emph{myriad} of asymptotically distribution-free goodness-of-fit tests.

Third, we have assumed that the null hypothesis corresponds to a family $\mathscr{F}_{\Theta}$ of \emph{stationary} point processes. 
In practice, one may wish to include \emph{nonstationary} models, such as Hawkes processes with time-varying baseline intensities. 
This extension is challenging and beyond our scope. 
Adapting our methodology would require a FCLT for nonstationary, non-ergodic martingales, that is, a nonstationary analogue of Lemma~\ref{FCLTmartingales}.
Moreover, the approach requires the availability of a central limit theorem for a suitable estimator; see \citet{ChenHall2013,Lavancier2021} for progress in this direction.

Finally, as noted in Section~\ref{sectionRTC}, if a consistent estimator of $\alpha_{\theta_0}Z$ could be constructed from the data, it would be possible to modify \eqref{diff transformed interarrival times} to develop an asymptotically distribution-free random-time-change-based goodness-of-fit test. 
However, solving this problem is far from straightforward due to the dependence of $\alpha_{\theta_0}Z$ on the unobservable quantity $\Lambda_{\theta_0}(t_{i}) - \Lambda_{\theta_0}(t_{i-1})$ in (\ref{diff transformed interarrival times}). 
Despite this challenge, this remains an intriguing direction for future research.

{\section*{Acknowledgements}
We are very grateful to Estate V. Khmaladze 
and seminar and conference participants at 
Sorbonne Université, 
the 2025 Conference on Stochastic Processes and their Applications in Wrocław, 
the 2026 Conference on Point Processes and their Applications in Lille,
the 2026 Conference on Statistical Methods for Dynamical Stochastic Models in Gothenburg,
and the University of Amsterdam
for stimulating discussions.
The research of Justin Baars and Roger Laeven has been funded in part by the Netherlands Organization for Scientific Research under an NWO VICI grant (2020--2027). 
The authors report that there are no competing interests to declare.
}


\bibliographystyle{agsm}
\bibliography{bibliography}

@article{Adelfio,
  author = {G. Adelfio and M. Chiodi},
  year = {2015},
  title = {{FLP} estimation of semi-parametric models for space-time point processes and diagnostic tools},
  journal = {Spatial Statistics},
  volume = {14},
  pages = {119--132}
}

@article{ALP14,
  author = {Y. A\"it-Sahalia and R. J. A. Laeven and L. Pelizzon},
  year = {2014},
  title = {Mutual excitation in {E}urozone sovereign {CDS}},
  journal = {J. Econom.},
  volume = {183},
  pages = {151--167}
}

@article{ACL15,
  author = {Y. A\"it-Sahalia and J. A. Cacho-Diaz and R. J. A. Laeven},
  year = {2015},
  title = {Modeling financial contagion using mutually exciting jump processes},
  journal = {J. Financ. Econ.},
  volume = {117},
  pages = {585--606}
}

@article{RTC9,
  author = {H. Albrecher and M. Bladt and D. Kortschak and F. Prettenthaler and T. Swierczynski},
  year = {2019},
  title = {Flood occurrence change-point analysis in the paleoflood record from {L}ake {M}ondsee ({NE} {A}lps)},
  journal = {Global and Planetary Change},
  volume = {178},
  pages = {65--76}
}

@book{Andersen,
  author = {P. K. Andersen and {\O}. Borgan and R. D. Gill and N. Keiding},
  year = {1993},
  title = {Statistical Models Based on Counting Processes},
  publisher = {Springer-Verlag},
  address = {New York},
  series = {Springer Series in Statistics}
}

@article{Bacry,
  author = {E. Bacry and S. Delattre and M. Hoffmann and J. F. Muzy},
  year = {2013},
  title = {Some limit theorems for {H}awkes processes and application to financial statistics},
  journal = {Stoch. Process. Appl.},
  volume = {123},
  pages = {2475--2499}
}

@book{Billingsley,
  author = {P. Billingsley},
  year = {1999},
  title = {Convergence of Probability Measures},
  publisher = {John Wiley \& Sons}
}

@misc{Bonnet,
  author = {A. Bonnet and C. Dion-Blanc and M. Sadeler Perrin},
  year = {2024},
  title = {Testing procedures based on maximum likelihood estimation for Marked {H}awkes processes},
  note = {\url{https://arxiv.org/abs/2410.05008}},
  howpublished = {Preprint}
}

@book{methods3,
  author = {P. Brémaud},
  year = {1981},
  title = {Point Processes and Queues: Martingale Dynamics},
  publisher = {Springer-Verlag},
  address = {New York}
}

@article{Can15,
  author = {S. U. Can and J. H. J. Einmahl and E. V. Khmaladze and R. J. A. Laeven},
  year = {2015},
  title = {Asymptotically distribution-free goodness-of-fit testing for tail copulas},
  journal = {Ann. Statist.},
  volume = {43},
  pages = {878--902}
}

@article{RTC5,
  author = {V. Chavez-Demoulin and J. A. McGill},
  year = {2012},
  title = {High-frequency financial data modeling using {H}awkes processes},
  journal = {J. Bank. Finance},
  volume = {36},
  pages = {3415--3426}
}

@article{ChenHall2013,
  author = {F. Chen and P. Hall},
  year = {2013},
  title = {Inference for a nonstationary self-exciting point process with an application in ultra-high frequency financial data modeling},
  journal = {J. Appl. Probab.},
  volume = {50},
  pages = {1006--1024}
}

@article{RTC4,
  author = {F. Chen and W. H. Tan},
  year = {2018},
  title = {Marked self-exciting point process modelling of information diffusion on {T}witter},
  journal = {Ann. Appl. Stat.},
  volume = {12},
  pages = {2175--2196}
}

@article{methods1,
  author = {R. A. Clements and F. P. Schoenberg and A. Veen},
  year = {2013},
  title = {Evaluation of space-time point process models using super-thinning},
  journal = {Environmetrics},
  volume = {23},
  number = {7},
  pages = {606--616}
}

@article{Clinet,
  author = {S. Clinet and W. T. M. Dunsmuir and G. W. Peters and K.-A. Richards},
  year = {2021},
  title = {Asymptotic distribution of the score test for detecting marks in {H}awkes processes},
  journal = {Stat. Inference Stoch. Process.},
  volume = {24},
  number = {3},
  pages = {635--668}
}

@book{DaleyVereJones,
  author = {D. J. Daley and D. Vere-Jones},
  year = {2003},
  title = {An Introduction to the Theory of Point Processes},
  edition = {2nd},
  volume = {I \& II},
  publisher = {Springer-Verlag},
  address = {New York}
}

@article{IMT1,
  author = {M. A. Delgado and J. Hidalgo and C. Velasco},
  year = {2005},
  title = {Distribution-free goodness-of-fit tests for linear processes},
  journal = {Ann. Statist.},
  volume = {33},
  pages = {2568--2609}
}

@article{IMT2,
  author = {H. Dette and B. Hetzler},
  year = {2009},
  title = {Khmaladze transformation of integrated variance processes with applications to goodness-of-fit testing},
  journal = {Math. Methods Statist.},
  volume = {18},
  pages = {97--116}
}

@article{ElAroui,
  author = {M.-A. El-Aroui},
  year = {2025},
  title = {On the use and misuse of time-rescaling to assess the goodness-of-fit of self-exciting temporal point processes},
  journal = {J. Appl. Stat.},
  volume = {52},
  pages = {2247--2270}  
}

@article{Hawkes,
  author = {A. G. Hawkes},
  year = {1971},
  title = {Spectra of some self-exciting and mutually exciting point processes},
  journal = {Biometrika},
  volume = {58},
  pages = {83--90}
}

@article{Hawkes2,
  author = {A. G. Hawkes and D. Oakes},
  year = {1974},
  title = {A cluster process representation of a self-exciting process},
  journal = {J. Appl. Probab.},
  volume = {11},
  pages = {493--503}
}

@article{HawkesSurvey22,
  author  = {A. G. Hawkes},
  journal = {The European Journal of Finance},
  title   = {Hawkes jump-diffusions and finance: {A} brief history and review},
  year    = {2022},
  pages   = {627-641},
  volume  = {28}
}

@article{RTC2,
  author = {P. Embrechts and T. Liniger and L. Lin},
  year = {2011},
  title = {Multivariate {H}awkes processes: {A}n application to financial data},
  journal = {J. Appl. Probab.},
  volume = {48},
  pages = {367--378}
}

@article{KlingVetter,
  author = {J. C. {Fontanesi Kling} and M. Vetter},
  year = {2026},
  title = {On goodness-of-fit testing for self-exciting point processes},
  journal = {Scand. J. Stat.},
  volume = {52},
  pages = {102--139}
}

@article{Hjort,
  author = {N. L. Hjort},
  year = {1986},
  title = {Bayes estimators and asymptotic efficiency in parametric counting process models},
  journal = {Scand. J. Stat.},
  volume = {13},
  pages = {63--85}
}

@article{Ikefuji,
  author = {M. Ikefuji and R. J. A. Laeven and J. R. Magnus and Y. Yue},
  year = {2022},
  title = {Earthquake risk embedded in property prices: Evidence from five {J}apanese cities},
  journal = {J. Am. Stat. Assoc.},
  volume = {117},
  pages = {82--93}
}

@book{JacodShiryaev,
  author = {J. Jacod and A. Shiryaev},
  year = {2003},
  title = {Limit Theorems for Stochastic Processes},
  publisher = {Springer-Verlag}
}

@article{Estate81,
  author = {E. V. Khmaladze},
  year = {1981},
  title = {A martingale approach in the theory of goodness-of-fit tests},
  journal = {Teor. Veroyatn. Primen.},
  volume = {26},
  pages = {246--265}
}

@article{Estate88,
  author = {E. V. Khmaladze},
  year = {1988},
  title = {An innovation approach to goodness-of-fit tests in $\mathbb{R}^m$},
  journal = {Ann. Statist.},
  volume = {16},
  pages = {1503--1516}
}

@article{Estate93,
  author = {E. V. Khmaladze},
  year = {1993},
  title = {Goodness of fit problem and scanning innovation martingales},
  journal = {Ann. Statist.},
  volume = {21},
  pages = {798--829}
}

@article{IMT3,
  author = {E. V. Khmaladze and H. L. Koul},
  year = {2004},
  title = {Martingale transforms goodness-of-fit tests in regression models},
  journal = {Ann. Statist.},
  volume = {32},
  pages = {995--1034}
}

@article{IMT4,
  author = {E. V. Khmaladze and H. L. Koul},
  year = {2009},
  title = {Goodness-of-fit problem for errors in nonparametric regression: Distribution-free approach},
  journal = {Ann. Statist.},
  volume = {37},
  pages = {3165--3185}
}

@article{IMT5,
  author = {R. Koenker and Z. Xiao},
  year = {2002},
  title = {Inference on the quantile regression process},
  journal = {Econometrica},
  volume = {70},
  pages = {1583--1612}
}

@article{IMT6,
  author = {R. Koenker and Z. Xiao},
  year = {2006},
  title = {Quantile autoregression},
  journal = {J. Am. Stat. Assoc.},
  volume = {101},
  pages = {980--990}
}

@article{RTC6,
  author = {T. J. Kwan and F. Chen and W. T. M. Dunsmuir},
  year = {2023},
  title = {Alternative asymptotic inference theory for a nonstationary {H}awkes process},
  journal = {J. Stat. Plan. Inference},
  volume = {227},
  pages = {75--90}
}

@article{RTC1,
  author = {M. Lallouache and D. Challet},
  year = {2016},
  title = {The limits of statistical significance of {H}awkes processes fitted to financial data},
  journal = {Quant. Finance},
  volume = {16},
  pages = {1--11}
}

@article{Lavancier2021,
  author = {F. Lavancier and A. Poinas and R. Waagepetersen},
  year = {2021},
  title = {Adaptive estimating function inference for non-stationary determinantal point processes},
  journal = {Scand. J. Stat.},
  volume = {48},
  pages = {87--107}
}

@misc{Lotz,
  author = {A. Lotz},
  year = {2024},
  title = {A sparsity test for multivariate {H}awkes processes},
  note = {Preprint, \url{https://arxiv.org/pdf/2405.08640}}
}

@misc{RTC3,
  author = {R. Luo and V. Krishnamurthy and E. Blasch},
  year = {2022},
  title = {Hawkes process modeling of block arrivals in {B}itcoin blockchain},
  note = {Preprint, \url{https://arxiv.org/abs/2203.16666}}
}

@article{IMT7,
  author = {I. W. McKeague and A. M. Nikabadze and Y. Q. Sun},
  year = {1995},
  title = {An omnibus test for independence of a survival time from a covariate},
  journal = {Ann. Statist.},
  volume = {23},
  pages = {450--475}
}

@article{Ogata,
  author = {Y. Ogata},
  year = {1978},
  title = {The asymptotic behaviour of maximum likelihood estimators for stationary point processes},
  journal = {Ann. Inst. Stat. Math.},
  volume = {30},
  pages = {243--261}
}

@article{methods2,
  author = {Y. Ogata},
  year = {1981},
  title = {On {L}ewis’ simulation method for point processes},
  journal = {IEEE Trans. Inf. Theory},
  volume = {IT-27},
  pages = {23--31}
}

@article{Ogata2,
  author = {Y. Ogata},
  year = {1988},
  title = {Statistical models for earthquake occurrences and residual analysis for point processes},
  journal = {J. Am. Stat. Assoc.},
  volume = {83},
  pages = {9--27}
}

@article{Ogata3,
  author = {Y. Ogata},
  year = {1998},
  title = {Space-time point-process models for earthquake occurrences},
  journal = {Ann. Inst. Stat. Math.},
  volume = {50},
  pages = {379--402}
}

@misc{VanPanhuis2018,
  author = {W. Van Panhuis and A. Cross and D. Burke},
  year = {2018},
  title = {Counts of {R}ocky {M}ountain spotted fever reported in {U}nited {S}tates of {A}merica: 1942--2016 (version 2.0, April 1, 2018)},
  note = {Project Tycho data release, \url{https://doi.org/10.25337/T7/ptycho.v2.0/US.186772009}}
}

@article{Puri,
  author = {M. L. Puri and P. D. Tuan},
  year = {1986},
  title = {Maximum likelihood estimation for stationary point processes},
  journal = {Proc. Natl. Acad. Sci. USA},
  volume = {83},
  pages = {541--545}
}

@article{MLEinfinitedimensional,
  author = {S. L. Rathbun},
  year = {1996},
  title = {Asymptotic properties of the maximum likelihood estimator for spatio-temporal point processes},
  journal = {J. Stat. Plan. Inference},
  volume = {51},
  pages = {55--74}
}

@article{Reinhart,
  author = {A. Reinhart},
  year = {2018},
  title = {A review of self-exciting spatio-temporal point processes and their applications},
  journal = {Statist. Sci.},
  volume = {33},
  pages = {299--318}
}

@article{ReynaudBouret,
  author = {P. Reynaud-Bouret and V. Rivoirard and F. Grammont and C. Tuleau-Malot},
  year = {2014},
  title = {Goodness-of-Fit Tests and Nonparametric Adaptive Estimation for Spike Train Analysis},
  journal = {J. Math. Neurosci.},
  volume = {4},
  pages = {1--41}
}

@article{Richards,
  author = {K. A. Richards and W. T. M. Dunsmuir and G. W. Peters},
  year = {2022},
  title = {Score test for marks in {H}awkes processes},
  journal = {Int. J. Data Sci. Anal.},
  volume = {24},
  pages = {1--16},
  note = {(3)}
}

@article{methods4,
  author = {F. P. Schoenberg},
  year = {1999},
  title = {Transforming spatial point processes into {P}oisson processes},
  journal = {Stoch. Process. Their Appl.},
  volume = {81},
  pages = {155--164}
}

@article{Schoenberg03,
  author = {F. P. Schoenberg},
  year = {2003},
  title = {Multi-dimensional residual analysis of point process models for earthquake occurrences},
  journal = {J. Am. Stat. Assoc.},
  volume = {98},
  pages = {789--795}
}

@article{Schoenberg,
  author = {F. P. Schoenberg and M. Hoffmann and R. J. Harrigan},
  year = {2019},
  title = {A recursive point process model for infectious diseases},
  journal = {Ann. Inst. Stat. Math.},
  volume = {71},
  pages = {1271--1287}
}

@article{IMT9,
  author = {W. Stute and S. Thies and L.-X. Zhu},
  year = {1998},
  title = {Model checks for regression: An innovation process approach},
  journal = {Ann. Stat.},
  volume = {26},
  pages = {1916--1934}
}

@article{SunCounting,
  author = {Y. Sun and R. C. Tiwari and J. N. Zalkikar},
  year = {2001},
  title = {Goodness of fit tests for multivariate counting process models with applications},
  journal = {Scand. J. Stat.},
  volume = {28},
  pages = {241--256},
  note = {(1)}
}

@article{RTC7,
  author = {I. M. Toke and F. Pomponio},
  year = {2012},
  title = {Modelling trades-through in a limit order book using {H}awkes processes},
  journal = {Economics – Open-Access, Open-Assessment E-J.},
  volume = {6},
  pages = {1--23}
}

@article{RTC8,
  author = {H. J. T. Unwin and I. Routledge and S. Flaxman and M.-A. Rizoiu and S. Lai and J. Cohen and D. J. Weiss and S. Mishra and S. Bhatt},
  year = {2021},
  title = {Using {H}awkes Processes to model imported and local malaria cases in near-elimination settings},
  journal = {PLoS Comput. Biol.},
  volume = {17},
  pages = {e1008830},
  note = {(4)}
}

@book{vdVaart,
  author = {A. W. van der Vaart},
  year = {1998},
  title = {Asymptotic Statistics},
  series = {Cambridge Ser. Statist. Probab. Math.},
  publisher = {Cambridge University Press}
}

@article{Veen,
  author = {A. Veen and F. P. Schoenberg},
  year = {2008},
  title = {Estimation of space-time Branching process models in seismology using an {EM}-type algorithm},
  journal = {J. Am. Stat. Assoc.},
  volume = {103},
  pages = {614--624}
}

@article{methods5,
  author = {D. Vere-Jones and F. P. Schoenberg},
  year = {2004},
  title = {Rescaling marked point processes},
  journal = {Austral. N. Z. J. Stat.},
  volume = {46},
  pages = {133--143},
  note = {(1)}
}

@book{Whitt,
  author = {W. Whitt},
  year = {2002},
  title = {Stochastic-Process Limits: An Introduction to Stochastic-Process Limits and Their Application to Queues},
  series = {Springer Ser. Oper. Res.},
  publisher = {Springer}
}

@article{Zhu,
  author = {L. Zhu},
  year = {2013},
  title = {Central limit theorem for nonlinear {H}awkes processes},
  journal = {J. Appl. Probab.},
  volume = {50},
  pages = {760--771}
}

\renewcommand{\theequation}{A\arabic{equation}}
\setcounter{equation}{0}

\renewcommand{\thetheorem}{A\arabic{theorem}}
\setcounter{theorem}{0}

\renewcommand{\thelemma}{A\arabic{lemma}}
\setcounter{lemma}{0}

\renewcommand{\theremark}{A\arabic{lemma}}
\setcounter{remark}{0}

\renewcommand{\thetable}{A\arabic{table}}
\renewcommand{\thefigure}{A\arabic{figure}}

\makeatletter
\renewcommand{\theHtable}{A\arabic{table}}   
\renewcommand{\theHfigure}{A\arabic{figure}} 
\makeatother

\setcounter{table}{0}
\setcounter{figure}{0}

\section*{Appendix A: Technical assumptions relegated from  Section~\ref{section def}}\label{app A}

In this appendix, we first provide some assumptions required for maximum likelihood estimation.
\begin{customthm}{C.1}\label{MLEass}\phantom.
\begin{enumerate}[label=(\roman*)]
\item The log-likelihood function $L_T(\theta)$ at time $T$, defined by
\begin{equation}
L_T(\theta):=\sum_{k=1}^d\left(\int_0^T\log\lambda_\theta^{(k)}(t)\, \mathrm dN^{(k)}(t)-\int_0^T\lambda_\theta^{(k)}(t)\, \mathrm dt\right),
\end{equation}
has a unique maximum a.s. 

\item For each $\theta\in\Theta$ and for all $i,j\in[m]$, $k\in[d]$, 
$\frac1{\lambda_\theta^{(k)}}\frac{\partial\lambda_\theta^{(k)}}{\partial\theta_i}\frac{\partial\lambda_\theta^{(k)}}{\partial\theta_j}\in L^2(\mathbb P_\theta)$. 
Moreover, for each $\theta\in\Theta$, the \emph{Fisher information matrix} $I(\theta)\in\mathbb R^{m\times m}$ with elements
\begin{equation}\label{Fisher eq}
I_{ij}(\theta):=\sum_{k=1}^d\mathbb E_\theta\left[\frac1{\lambda_\theta^{(k)}}\frac{\partial\lambda_\theta^{(k)}}{\partial\theta_i}\frac{\partial\lambda_\theta^{(k)}}{\partial\theta_j}\right]
\end{equation}
is nonsingular.
\item $\lambda_\theta\in C^3(\Theta)$ with $\displaystyle\frac{\partial^2\lambda_\theta^{(k)}}{\partial\theta_i\partial\theta_j}\in L^2(\mathbb P_\theta)$ for all $k\in[d]$, $i,j\in[m]$.
\item 
Validity of stochastic approximations as detailed below.
\end{enumerate}
\end{customthm}

Next, this appendix contains some technical assumptions needed for the stochastic approximations of conditional intensities depending on the complete history by those depending only on the history on $(0,t)$.
For the following set of assumptions, let
$
\mathcal H_{0,t-}:=\sigma(N(s):s\in[0,t)),
$
and let $\lambda_\theta^*$ be any $(\mathcal H_t)_{t\in\mathbb R}$-predictable function such that
\begin{equation*}
\lambda_\theta^*(t)=\lim_{\Delta t\downarrow0}\frac1{\Delta t}\mathbb P_\theta\left(N[t,t+\Delta t)>0|\mathcal H_{0,t-}\right).
\end{equation*}

\begin{customthm}{C.1(iv)}\label{MLEassv}\phantom.
\begin{enumerate}[label=(\alph*)]
\item $\mathbb E[\sup_{\delta\in(0,1]}\delta^{-1}(N[0,\delta])^2]<\infty$.
\item For each $\theta\in\Theta$, there exists a neighborhood $U=U(\theta)\subset\Theta$ of $\theta$ such that if \begin{align*}
H(t):=\max_{\substack{i,j,k\in[m]\\\ell\in[d]}}\sup_{\theta'\in U}\left|\frac{\partial^3\lambda_{\theta'}^{(\ell)}(t)}{\partial\theta_i\partial\theta_j\partial\theta_k}\right|,\quad
G(t):=\max_{\substack{i,j,k\in[m]\\\ell\in[d]}}\sup_{\theta'\in U}\left|\frac{\partial^3\log\lambda_{\theta'}^{(\ell)}(t)}{\partial\theta_i\partial\theta_j\partial\theta_k}\right|,
\end{align*}
then $\mathbb E_\theta[H(0)]<\infty$ and $\mathbb E_\theta[\|\lambda_{\theta}(0)\|^2G(0)^2]<\infty$.
\item Under each $\mathbb P_\theta\in\mathscr F_\Theta$, and for all $i,j\in[m]$, $\ell\in[d]$, the following tend to zero in probability as $t\to\infty$:
$$ 
\lambda_\theta^{(\ell)}-(\lambda_\theta^*)^{(\ell)},\quad \frac{\partial\lambda_\theta^{(\ell)}}{\partial\theta_i}-\frac{\partial(\lambda_\theta^*)^{(\ell)}}{\partial\theta_i}\quad\text{and}\quad\frac{\partial^2\lambda_\theta^{(\ell)}}{\partial\theta_i\theta_j}-\frac{\partial^2(\lambda_\theta^*)^{(\ell)}}{\partial\theta_i\theta_j}.
$$
\item For each $\theta\in\Theta$, there is some $\alpha>0$ such that the following have finite, uniformly  (w.r.t.\ $t$) bounded $(2+\alpha)$th $\mathbb P_\theta$-moments:
$$
\frac{\lambda_\theta^{(\ell)}}{(\lambda_\theta^*)^{(\ell)}},\quad\frac1{(\lambda_\theta^*)^{(\ell)}}\frac{\partial(\lambda^*_\theta)^{(\ell)}}{\partial\theta_i}\frac{\partial(\lambda^*_\theta)^{(\ell)}}{\partial\theta_j}\quad\text{and}\quad\frac{\partial^2(\lambda^*_\theta)^{(\ell)}}{\partial\theta_i\partial\theta_j},\quad i,j\in[m], \ell\in[d].
$$
\item For each $\theta\in\Theta$, $i\in[m]$, $\ell\in[d]$, as $T\to\infty$,
\begin{align*}
&\mathbb E_\theta\left[\frac1{\sqrt T}\int_0^T\left|\frac{\partial\lambda_\theta^{(\ell)}}{\partial\theta_i}-\frac{\partial(\lambda_\theta^*)^{(\ell)}}{\partial\theta_i}\right|\ \mathrm dt\right]\to0;\\
&\mathbb E_\theta\left[\frac1{\sqrt T}\int_0^T\left|\lambda_\theta^{(\ell)}-(\lambda_\theta^*)^{(\ell)}\right|\ \left|\frac1{(\lambda_\theta^*)^{(\ell)}}\frac{\partial(\lambda^*_\theta)^{(\ell)}}{\partial\theta_i}\right| \mathrm dt\right]\to0.
\end{align*}
\item For each $\theta\in\Theta$, there exists a neighborhood $U=U(\theta)\subset\Theta$ of $\theta$ such that 
$$\max_{\substack{i,j,k\in[m]\\\ell\in[d]}}\sup_{\theta'\in U}\left|\frac{\partial^3\lambda_{\theta'}^{(\ell)}(t)}{\partial\theta_i\partial\theta_j\partial\theta_k}-\frac{\partial^3(\lambda_{\theta'}^*)^{(\ell)}(t)}{\partial\theta_i\partial\theta_j\partial\theta_k}\right|\stackrel{\mathbb P_\theta}\to0\text{ as }t\to\infty.$$
\item For each $\theta\in\Theta$, there exists a neighborhood $U=U(\theta)\subset\Theta$ of $\theta$ and some $\alpha>0$ such that 
$$
\max_{\substack{i,j,k\in[m]\\\ell\in[d]}}\sup_{\theta'\in U}\frac{\partial^3(\lambda_{\theta'}^*)^{(\ell)}}{\partial\theta_i\partial\theta_j\partial\theta_k}\quad\text{and}\quad \max_{\substack{i,j,k\in[m]\\\ell\in[d]}}\sup_{\theta'\in U}\frac{\partial^3\log(\lambda_{\theta'}^*)^{(\ell)}}{\partial\theta_i\partial\theta_j\partial\theta_k}
$$ have finite, uniformly (w.r.t.\ $t$) bounded  $(2+\alpha)$th $\mathbb P_\theta$-moments.
\end{enumerate}
\end{customthm}

\section*{Appendix B: Proofs}\label{app B} 
In this appendix, we provide the proofs of our results contained in Sections~\ref{Section3} and~\ref{Section4.2}.

\section*{Proofs of Section~\ref{Section3}}
\subsection*{Step 1: FCLT for the compensated point process}
To prove Theorem~\ref{step1thm}, we need a multivariate analog of \citet[Theorem 18.3]{Billingsley}. 
In order to prove this, we invoke the following lemma, establishing a Cramér-Wold device for stochastic processes.

\begin{lemma}[Cramér-Wold device for stochastic processes]\label{cramerwold}
Let $(W_n)_{n\in\mathbb N}\subset D([0,1],\mathbb R^d)$ be a sequence of $d$-dimensional stochastic processes such that 
\begin{equation}\label{CWdevice}
\text{for all }c\in\mathbb R^d:\quad c^\top W_n\stackrel{d}\longrightarrow c^\top W,
\end{equation} on $D([0,1],\mathbb R)$ equipped with the Skorokhod $J_1$ topology, and where $W$ is a Wiener process on $\mathbb R^d$. Then it holds that 
$W_n\stackrel{d}\longrightarrow W$ on $D([0,1],\mathbb R^d)$ equipped with the \emph{strong} Skorokhod $J_1$ topology.
\end{lemma}
\begin{proof}
First, by the standard Cramér-Wold device, the condition \eqref{CWdevice} implies that the finite-dimensional distributions of $W_n$ converge to those of $W$.

Next, condition \eqref{CWdevice} implies that the one-dimensional marginals converge \emph{to a continuous limit}, meaning that we have weak convergence in $D([0,1],\mathbb R)$ equipped with the \emph{uniform topology}. 
Because of this weak convergence, we have tightness (w.r.t.\ the uniform topology) of the one-dimensional marginals. 
In turn, by \citet[Theorem~11.6.7]{Whitt}, this implies tightness of the processes $W_n$ themselves; this tightness is w.r.t.\ the product topology on $D([0,1],\mathbb R)^d$ obtained out of uniform topologies on $D([0,1],\mathbb R)$. Note that here the $\|\cdot\|_\infty$-norm on $\mathbb R^d$ is used. 

Using Prohorov's theorem, those two conditions imply that $W_n\stackrel{d}\longrightarrow W$ on $D([0,1],\mathbb R)^d$ equipped with the product of uniform topologies, $WU$. 
We now prove that this is equal to the strong uniform topology on $D([0,1],\mathbb R^d)$, $SU$. 

Indeed, the product topology $WU$ is generated by the subbase consisting of the sets 
$$
\pi_i^{-1}\left\{x_i\in D([0,1],\mathbb R):\|x_i-y_i\|_\infty<\epsilon\right\},\quad y_i\in D([0,1],\mathbb R), \epsilon>0,i\in[d],
$$
where $\pi_i$ is the projection operator onto the $i$th coordinate. Since 
$$
\pi_i^{-1}\left\{x_i\in D([0,1],\mathbb R):\|x_i-y_i\|_\infty<\epsilon\right\}=\bigcup_{\substack{z\in D([0,1],\mathbb R^d)\\z_i=y_i}}\left\{x\in D([0,1],\mathbb R^d):\|x-z\|_\infty<\epsilon\right\}
$$
lies in the uniform topology on $D([0,1],\mathbb R^d)$, it follows that $WU\subset SU$.
Next, note that SU on $D([0,1],\mathbb R^d)$ is generated by the sets
$$
\left\{x\in D([0,1],\mathbb R^d):\|x-y\|_\infty<\epsilon\right\},\quad y\in D([0,1],\mathbb R^d),\epsilon>0.
$$
Hence, the inclusion $WU\supset SU$ follows by the equality 
$$
\left\{x\in D([0,1],\mathbb R^d):\|x-y\|_\infty<\epsilon\right\}=\bigcap_{i\in[d]}\pi_i^{-1}\left\{x_i\in D([0,1],\mathbb R):\|x_i-y_i\|_\infty<\epsilon\right\}.
$$

To conclude, we have $W_n\stackrel{d}\longrightarrow W$ on $D([0,1],\mathbb R)^d$ equipped with the strong uniform topology $SU$.
This implies convergence in the $SJ_1$ topology.
\end{proof}

\begin{remark}
As is evident from the proof, Lemma \ref{cramerwold} allows for generalizations to stochastic processes for which all the finite-dimensional distributions are Borel probability measures.
\end{remark}

The following result, establishing an FCLT for multivariate martingales, is a multivariate extension of \citet[Theorem 18.3]{Billingsley}. 

\begin{lemma}[FCLT for multivariate martingales]\label{FCLTmartingales}
Consider a stationary and ergodic two-sided $d$-variate martingale difference sequence $(\xi_n)_{n\in\mathbb Z}$, where $\xi_n$ takes values in $\mathbb R^d$ and satisfies $\mathbb E[\xi_n|\mathcal F_{n-1}]=0$, with $\mathcal F_n:=\sigma(\xi_k:k\leq n)$. Write $\Xi:=(\Xi_{ij})_{i,j\in[d]}:=(\mathbb E[\xi^{(i)}\xi^{(j)}])_{i,j\in[d]}\in\mathbb R^{d\times d}$, where $\xi^{(i)}$ denotes the $i$th coordinate of the $d$-dimensional generic random vector $\xi$. The stochastic process $t\mapsto X_t^n:=\sum_{k\leq nt}\xi_k/\sqrt n$ converges weakly on $D([0,1],\mathbb R^d)$ equipped with the strong Skorokhod $J_1$ topology to a $d$-variate Wiener process $W$ with covariance matrix $u\Xi$.
\end{lemma}
\begin{proof}
By Lemma \ref{cramerwold}, it suffices to show that for arbitrary $c\in\mathbb R^d$ it holds that $c^\top X^n\stackrel{d}\longrightarrow c^\top W$ on $D[0,1]$ equipped with the Skorokhod $J_1$ topology. 

To this end, note that $(c^\top \xi_n)_{n\in\mathbb Z}$ is a univariate martingale difference sequence w.r.t.\ $(\mathcal F_n)_{n\in\mathbb Z}$, with second moment 
$$
\mathbb E\left[\left(c^\top \xi_k\right)^2\right]=\sum_{i=1}^n\sum_{j=1}^nc_ic_j\mathbb E\left[\xi_n^{(i)}\xi_n^{(j)}\right]=\sum_{i=1}^n\sum_{j=1}^nc_ic_j\Xi_{ij}=c^\top\Xi c\geq0.
$$

The convergence $c^\top X^n\stackrel{d}\longrightarrow c^\top W$ in the case $\mathbb E\left[\left(c^\top \xi_k\right)^2\right]=0$ is trivial, hence we may assume from now on that $\mathbb E\left[\left(c^\top \xi_k\right)^2\right]>0$. 
Then, by \citet[Theorem 18.3]{Billingsley}, it follows that $c^\top X^n$ converges weakly on $D([0,1],\mathbb R)$ equipped with the Skorokhod $J_1$ topology to a Wiener process with covariance $uc^\top \Xi c$, i.e., to a limit $c^\top W$, where $W$ is a $d$-variate Wiener process with covariance $u\Xi$.
\end{proof}

We are now ready to provide the proof of Theorem~\ref{step1thm}.

\textit{Proof of Theorem~\ref{step1thm}}.
The result follows by an application of 
Lemma~\ref{FCLTmartingales} to $\xi_n=N[n,n+1)-\int_{n}^{n+1}\lambda_{\theta_0}(s)\ \mathrm ds$, combined with an estimate of continuous-time quantities by discrete-time quantities analogous to \citet[eqn.\ (2.19)]{Zhu}.
We use the calculation
\begin{align*}
\mathbb E_{\theta_0}\left[\left(N^{(i)}[0,1]-\int_0^{1}\lambda_{\theta_0}^{(i)}(s)\ \mathrm ds\right)\left(N^{(j)}[0,1]-\int_0^{1}\lambda_{\theta_0}^{(j)}(s)\ \mathrm ds\right)\right]=
\begin{cases}
\mathbb E_{\theta_0}\left[N^{(i)}[0,1]\right]\quad&\text{if }i=j,\\
0\quad&\text{if }i\neq j,
\end{cases}
\end{align*}
which follows from \citet[Proposition~14.1.VIII]{DaleyVereJones}. 
$\hfill\Box$

\subsection*{Step 2: Limit theorem for the compensators}
\textit{Proof of Theorem~\ref{step2thm}}.
Our proof is structured as follows.
Assumptions~\ref{CLTass} grant the CLT $\sqrt T(\theta_0-\hat\theta_T)\stackrel{d}\to Z$. This allows us to perform a Taylor expansion on the integrand of the prelimit from (\ref{step2limit}): $\lambda_{\theta_0}(s)-\lambda_{\hat\theta_T}(s)=D\lambda_{\theta_0}(s)(\theta_0-\hat\theta_T)+o_{\mathbb P_{\theta_0}}(|\theta_0-\hat\theta_T|)$. We use this to rewrite the prelimit from (\ref{step2limit}) to
\begin{equation}\label{Taylor integral}
\frac1{\sqrt T}\int_0^{uT}\left(\lambda_{\theta_0}(s)-\lambda_{\hat\theta_T}(s)\right)\ \mathrm ds=u\frac1{uT}\int_0^{uT}D\lambda_{\theta_0}(s)\ \mathrm ds\sqrt T(\theta_0-\hat\theta_T)+o_{\mathbb P_{\theta_0}}(1).
\end{equation}
We prove the result by showing that the $i$th coordinate of the right-hand side of (\ref{Taylor integral}) converges uniformly in probability to the $i$th coordinate of the right-hand side of (\ref{step2limit}). From this, convergence in the Skorokhod topology follows. 

Denote the $k$th row of $\alpha_{\theta_0}$ by $\alpha_{\theta_0}^{(k)}$, i.e.,
$$
\alpha_{\theta_0}^{(k)}:=\lim_{T\to\infty}\frac1T\int_0^T\frac{\partial \lambda_\theta^{(k)}(s)}{\partial\theta}\ \mathrm ds\bigg|_{\theta=\theta_0}\in\mathbb R^m.
$$
We need to show that, for each $k\in[d]$,
\begin{equation}\label{step2limit coordinatewise}
u\frac1{uT}\int_0^{uT}\frac{\partial \lambda_\theta^{(k)}(s)}{\partial\theta}\ \mathrm ds\bigg|_{\theta=\theta_0}\sqrt T(\theta_0-\hat\theta_T)\to u\alpha_{\theta_0}^{(k)}Z,
\end{equation}
uniformly in probability. 
The second factor $\sqrt T(\theta_0-\hat\theta_T)$ is not dependent on $u$ and converges to $Z$, by Assumptions~\ref{CLTass}, hence uniform convergence in probability $\sqrt T(\theta_0-\hat\theta_T)\to Z$ is immediate. Indeed, by the Skorokhod representation theorem, there exists a version of $\sqrt T(\theta_0-\hat\theta_T)$ which converges \emph{in probability} to a random variable with distribution $Z$. 

For the first factor, 
$$
u\frac1{uT}\int_0^{uT}\frac{\partial \lambda_\theta^{(k)}(s)}{\partial\theta}\bigg|_{\theta=\theta_0}\, \mathrm ds,
$$
note that by Assumption \ref{assB}(iii), $\|\partial\lambda_\theta^{(k)}(0)/\partial\theta|_{\theta=\theta_0}\|$ has finite expectation, hence by stationarity (Assumption~\ref{assA}(i)), the same holds for $\|\partial\lambda_\theta^{(k)}(s)/\partial\theta|_{\theta=\theta_0}\|$, $s\geq0$. 
In fact, it follows that $\mathbb E_{\theta_0}\|\partial\lambda_\theta^{(k)}(s)/\partial\theta|_{\theta=\theta_0}\|$ is independent of time $s$. 
By Fubini’s theorem, we infer that
$$
\mathbb E_{\theta_0}\left[\frac1T\int_0^{uT}\left\|\frac{\partial \lambda_\theta^{(k)}(s)}{\partial\theta}\bigg|_{\theta=\theta_0}\right\|\, \mathrm ds\right]=\frac1T\int_0^{uT}\mathbb E_{\theta_0}\left\|\frac{\partial \lambda_\theta^{(k)}(s)}{\partial\theta}\bigg|_{\theta=\theta_0}\right\|\, \mathrm ds=u\mathbb E_{\theta_0}\left\|\frac{\partial \lambda_\theta^{(k)}(0)}{\partial\theta}\bigg|_{\theta=\theta_0}\right\|.
$$
Fix $\delta,\epsilon>0$. For $u\in[0,1]$, $T>0$, an application of Markov’s inequality yields that with probability of at least $1-\delta/2$ it holds for all $u'\in[0,u]$ that 
$$
\frac1T\int_0^{u'T}\left\|\frac{\partial \lambda_\theta^{(k)}(s)}{\partial\theta}\bigg|_{\theta=\theta_0}\right\|\ \mathrm ds\leq\frac2\delta u\mathbb E_{\theta_0}\left\|\frac{\partial \lambda_\theta^{(k)}(0)}{\partial\theta}\bigg|_{\theta=\theta_0}\right\|.
$$
Note that this bound is independent of $T>0$. Hence, selecting
$$
u^{(k)}=\frac\epsilon4\left(\delta\left(\mathbb E_{\theta_0}\left\|\frac{\partial \lambda_\theta^{(k)}(0)}{\partial\theta}\bigg|_{\theta=\theta_0}\right\|\right)^{-1}\wedge\frac1{\|\alpha_{\theta_0}^{(k)}\|}\right)
$$
shows that with probability of at least $1-\delta/2$
$$
\sup_{u\in[0,u^{(k)}]}\left\|
\frac1{T}\int_0^{uT}\frac{\partial \lambda_\theta^{(k)}(s)}{\partial\theta}\bigg|_{\theta=\theta_0}\ \mathrm ds-u\alpha_{\theta_0}^{(k)}
\right\|\leq\epsilon.
$$

Given $u^{(k)}$, we use Assumption~\ref{assB}(vi) to select $T^{(k)}$ such that for $T\geq T^{(k)}$ and $u\in[u^{(k)},1]$ it holds with probability of at least $1-\delta/2$ that
$$
\left\|u\left(\frac1{uT}\int_0^{uT}\frac{\partial \lambda_\theta^{(k)}(s)}{\partial\theta}\bigg|_{\theta=\theta_0}\ \mathrm ds-\alpha_{\theta_0}^{(k)}\right)
\right\|\leq\epsilon.
$$
We have shown that on an event of probability at least $1-\delta$, for $T\geq T^{(k)}$ it holds that
$$
\sup_{u\in[0,1]}\left\|
\frac1{T}\int_0^{uT}\frac{\partial \lambda_\theta^{(k)}(s)}{\partial\theta}\bigg|_{\theta=\theta_0}\ \mathrm ds-u\alpha_{\theta_0}^{(k)}
\right\|\leq\epsilon.
$$
In other words, $\frac1{T}\int_0^{uT}\frac{\partial \lambda_\theta^{(k)}(s)}{\partial\theta}\big|_{\theta=\theta_0}\ \mathrm ds$ converges uniformly in probability to $u\alpha_{\theta_0}^{(k)}$.

Next, by a simple union bound similar to the one used in the proof of \citet[Theorem 2.7(vi)]{vdVaart}, it follows that
$$
\left(\frac1{T}\int_0^{uT}\frac{\partial \lambda_\theta^{(k)}(s)}{\partial\theta}\bigg|_{\theta=\theta_0}\ \mathrm ds,\sqrt T(\theta_0-\hat\theta_T)\right)\to\left(u\alpha_{\theta_0}^{(k)},Z\right),
$$
uniformly in probability. By an application of the continuous mapping theorem (to the function $g(x,y)=x\cdot y$) it follows that (\ref{step2limit coordinatewise}) holds, uniformly over $u\in[0,1]$, in probability.
We extend to (\ref{step2limit}) by selecting $u^*=\min_{k\in[d]}u^{(k)}$ and $T^*=\max_{k\in[d]}T^{(k)}$. 

Note that uniform convergence in probability amounts to uniform convergence on a set of probability $1-\delta/2$. 
Therefore, conditional on a high probability set, convergence in the Skorokhod $J_1$ topology holds, hence the tightness conditions for the $J_1$ topology hold. 
Conditional on a set of probability $1-\delta/2$, the probabilities appearing in the tightness conditions can be bounded by $\delta/2$. 
Therefore, they can even be bounded by $\delta$. 
Hence, uniform convergence in probability implies convergence in the Skorokhod $J_1$ topology. 
$\hfill\Box$

\subsection*{Step 3: Transforming $\eta_{\theta_0}$ into a standard Wiener process}

\textit{Proof of Theorem~\ref{step3thm}}.
Each coordinate of the limit $\eta_{\theta_0}$ found in Corollary \ref{step1+2} is of the general form treated in \citet[Section 3]{Can15}. Fix a coordinate $k\in[d]$. 
We can specify \citet[Theorem 3.1]{Can15} to univariate time frames by taking functions independent of the second coordinate. 
We identify the objects appearing in that theorem. 
We take the scanning family $A_v=[0,v)$ on $[0,1]$. 
We have $R(u)=\mu_{\theta_0}^{(k)}u$ and $$Q(B)=\int_B\, \mathrm dQ(u)=\int_B\, \mathrm du=\mathrm{Leb}(B)=\int_Bq(u)\, \mathrm dR(u)=\mu_{\theta_0}^{(k)}\int_Bq(u)\, \mathrm du,$$ hence $q(u)=1/\mu_{\theta_0}^{(k)}$. 
Next, we have 
$$
I(A_v^c)=\int_{A_v^c}q(u)^2\ \mathrm dR(u)=\int_v^1\frac1{(\mu_{\theta_0}^{(k)})^2}\mu_{\theta_0}\ \mathrm du=\frac{1-v}{\mu_{\theta_0}^{(k)}}.
$$
Note that $Q(B\cap A_{\mathrm dv})=Q(B\cap\{\mathrm dv\})=\mathbf1\{v\in B\}\ \mathrm dv$. Then \citet[Theorem 3.1]{Can15}, Theorem 3.1, gives that
\begin{equation}\label{Can eq}
\eta_{\theta_0}^{(k)}(B)-\int_B\frac{\eta_{\theta_0}^{(k)}(1)-\eta_{\theta_0}^{(k)}(v)}{1-v}\ \mathrm dv
\end{equation} is a Wiener process with variance $\mu_{\theta_0}^{(k)}u$, implying that $\mathscr T_{\theta_0}(\eta_{\theta_0})$ is a standard $d$-dimensional Wiener process. 

Indeed, $\mathscr T_{\theta_0}(\eta_{\theta_0})$ is a $d$-dimensional Gaussian process with zero mean, hence its probabilistic behavior is characterized by its covariance structure $C_{ij}(u):=\mathrm{Cov}\left(\mathscr T_{\theta_0}(\eta_{\theta_0}))_i(u),(\mathscr T_{\theta_0}(\eta_{\theta_0}))_j(u)\right)$, $i,j\in[d]$, $u\in[0,1]$. 
Using linearity of $\mathscr T_{\theta_0}^{(i)}$ and $\mathscr T_{\theta_0}^{(j)}$ in $\eta$ together with the definition on $\eta_{\theta_0}$ given in (\ref{step1+2 eq}), we expand
\begin{align*}
C_{ij}(u)&=\mathrm{Cov}\left(\mathscr T_{\theta_0}^{(i)}(W^{(i)}_{\theta_0}(u)),\mathscr T_{\theta_0}^{(j)}(W^{(j)}_{\theta_0}(u))\right)+\mathrm{Cov}\left(\mathscr T_{\theta_0}^{(i)}(W^{(i)}_{\theta_0}(u)),\mathscr T_{\theta_0}^{(j)}(u\alpha_{\theta_0}^{(j)}Z)\right)\\
&{\phantom=}\ +\mathrm{Cov}\left(\mathscr T_{\theta_0}^{(i)}(u\alpha_{\theta_0}^{(i)}Z),\mathscr T_{\theta_0}^{(j)}(W^{(j)}_{\theta_0}(u))\right)+\mathrm{Cov}\left(\mathscr T_{\theta_0}^{(i)}(u\alpha_{\theta_0}^{(i)}Z),\mathscr T_{\theta_0}^{(j)}(u\alpha_{\theta_0}^{(j)}Z))\right).
\end{align*}
For $i\neq j$, the first term vanishes because of independence between $W^{(i)}_{\theta_0}$ and $W^{(j)}_{\theta_0}$, while the remaining three terms vanish because $\mathscr T_{\theta_0}^{(i)}(u\alpha_{\theta_0}^{(k)}Z)\equiv0$ for each $k\in[d]$.
Therefore, $C_{ij}(u)\equiv0$.
$\hfill\Box$

\subsection*{Step 4: Replacing $\theta_0$ by $\hat\theta_T$}
\textit{Proof of Theorem \ref{step4thm}}.
The proof consists of two parts. First, we show that $\mu_{\hat\theta_T}\to\mu_{\theta_0}$ as $T\to\infty$. Second, we show that
$$
\tilde{\mathscr T}:=\mathrm{diag}\left(\sqrt{\mu_{\theta_0}^{(1)}},\ldots,\sqrt{\mu_{\theta_0}^{(d)}}\right)\mathscr T_\theta
$$ satisfies $\tilde{\mathscr T}(\hat\eta^{(T)})\to\tilde{\mathscr T}(\eta_{\theta_0})$ uniformly on $[0,\tau]\ni u$ in probability as $T\to\infty$; note that $\tilde{\mathscr T}$ could also be defined through the formula
$$
\tilde{\mathscr T}(\eta)(B):=\left\{\eta^{(k)}(B)-\int_B\frac{\eta^{(k)}(1)-\eta^{(k)}(v)}{1-v}\ \mathrm dv\right\}_{k\in[d]},\quad B\in\mathcal B[0,1].
$$

For the first part, by stationarity (Assumption~\ref{assA}(i)), we have that $\mu_{\theta_0}^{(k)}=\mathbb E_{\theta_0}[\lambda_{\theta_0}^{(k)}(0)]$. 
Furthermore, by Assumption \ref{CLTass}, it holds that $\hat\theta_T\to\theta_0$ in probability, as $T\to\infty$. 
By passing to a subsequence, this convergence can be taken a.s. 
Hence, by Assumption \ref{assB}(iii) and the continuous mapping theorem, we obtain $\lambda_{\hat\theta_T}^{(k)}(0)\to\lambda_{\theta_0}^{(k)}(0)$ a.s., as $T\to\infty$. 
By Assumption \ref{assB}(v) and dominated convergence, it then follows that $\mu_{\hat\theta_T}^{(k)}\to\mu_{\theta_0}^{(k)}$, as $T\to\infty$.

Next, by Assumption~\ref{assB}(iii), it holds that $\mu_{\theta_0}^{(k)}>0$, hence to prove the theorem it remains to prove the second part. 
To this end, we apply the Skorokhod representation theorem to the result of Corollary~\ref{step1+2} to obtain a probability space supporting probabilistic equivalent versions of $\hat\eta^{(T)}$ and $\eta_{\theta_0}$ satisfying
\begin{equation}\label{Skorokhod rep}
\sup_{u\in[0,\tau]}\left\|\hat\eta^{(T)}(u)-\eta_{\theta_0}(u)\right\|_\infty\to0\quad\text{a.s.}
\end{equation}
We will work on this probability space. 

To prove that $\tilde{\mathscr T}(\hat\eta^{(T)})\to\tilde{\mathscr T}(\eta_{\theta_0})$ uniformly on $[0,\tau]\ni u$ in probability, it suffices to prove for $k\in[d]$ that
\begin{equation}\label{step4 ucp}
\sup_{u\in[0,\tau]}\left|\int_0^u\frac{\Delta^{(T)}_k(1)-\Delta^{(T)}_k(v)}{1-v}\ \mathrm dv\right|\to0
\end{equation}
in probability, as $T\to\infty$, i.e., to prove ucp convergence. 
Here, $\Delta^{(T)}:=\hat\eta^{(T)}-\eta_{\theta_0}$ denotes the difference between the empirical process from Step 2 and its limit. 
Since $\frac1{1-v}$ is bounded on $[0,\tau]$, we simply infer the convergence in (\ref{step4 ucp}) from (\ref{Skorokhod rep}). 
$\hfill\Box$

\section*{Proofs of Section~\ref{Section4.2}}

\textit{Proof of Lemma~\ref{thm5lemma}}.
The maximum likelihood estimator $\hat\theta_T$ is consistent in probability for the ``least-false estimator'', or ``oracle'', $\theta^*\in\Theta$; see \citet[Section~2.2]{Hjort}. 
In particular, a CLT $\sqrt T(\theta^*-\hat\theta_T)\stackrel{d}\to Z$, as $T\to\infty$, still holds; cf.\ Assumptions \ref{MLEass}.



Select $\tau\in(0,1)$ as in Theorem~\ref{step4thm} and work on $[0,\tau]\ni u$. Recall that $\hat W^{(T)}:=\mathscr T_{\hat\theta_T}(\hat\eta^{(T)})$, where
\begin{align}\nonumber
\hat\eta^{(T)}(u)&=\frac1{\sqrt T}\left(N(uT)-\int_0^{uT}\lambda(s)\ \mathrm ds\right)+\frac1{\sqrt T}\left(\int_0^{uT}\lambda(s)\ \mathrm ds-\int_0^{uT}\lambda_{\theta^*}(s)\ \mathrm ds\right)\\
&\label{Thm5decomp}{\phantom=}\ +\frac1{\sqrt T}\left(\int_0^{uT}\lambda_{\theta^*}(s)\ \mathrm ds-\int_0^{uT}\lambda_{\hat\theta_T}(s)\ \mathrm ds\right).
\end{align}
By the CLT $\sqrt T(\theta^*-\hat\theta_T)\stackrel{d}\to Z$ and an argument analogous to the proof of Theorem \ref{step2thm}, the third term on the right-hand side of (\ref{Thm5decomp}) converges to a shift term linear in $u$, with random coefficient, which is annihilated by the innovation martingale transformation  $\mathscr T_{\hat\theta_T}$.
Next, the transformation (through $\mathscr T_{\hat\theta_T}$) of the first term on the right-hand side of (\ref{Thm5decomp})  converges weakly on $D([0,\tau],\mathbb R^d)$ to a standard Wiener process $W_1$, as $T\to\infty$. 

In order to be able to describe the deviation from the null $\lambda\notin\mathscr L_\Theta=\{\lambda_\theta:\theta\in\Theta\}$, we consider the contribution of the second term of (\ref{Thm5decomp})  after transformation through $\mathscr T_{\hat\theta_T}$. Note that 
\begin{equation}\label{trans 2nd term}
\mathscr T_{\hat\theta_T}\left(u\mapsto\frac1{\sqrt T}\left(\int_0^{uT}\lambda(s)\, \mathrm ds-\int_0^{uT}\lambda_{\theta^*}(s)\, \mathrm ds\right)\right)
\end{equation}
has, after a change of variables for the second integral, $k$th component equal to 
\begin{align}\nonumber
&\phantom{=}\frac1{\sqrt T\sqrt{\mu_{\hat\theta_T}^{(k)}}}\left(\int_0^{uT}g_k(s)\, \mathrm ds-\int_0^{uT}\frac1{T-v}\int_{v}^Tg_k(s)\, \mathrm ds\, \mathrm dv\right)\\&=\frac1{\sqrt T\sqrt{\mu_{\hat\theta_T}^{(k)}}}\int_0^{uT}\left(g_k(s)-\bar g_k^{T}(s)\right)\, \mathrm ds,\label{Dk}
\end{align}
where $g_k(s):=\lambda^{(k)}(s)-\lambda_{\theta^*}^{(k)}(s)$ and $\bar g_k^{T}(s):=(T-s)^{-1}\int_{s}^Tg_k(t)\ \mathrm dt$. 
Decompose (\ref{Dk}) into 
\begin{equation}\label{Dkdecomp}
        \frac1{\sqrt T\sqrt{\mu_{\hat\theta_T}^{(k)}}}\int_0^{uT}\left(g_k(s)-\mathbb E[g_k(0)]\right)\, \mathrm ds+
    \frac1{\sqrt T\sqrt{\mu_{\hat\theta_T}^{(k)}}}\int_0^{uT}\left(\mathbb E[g_k(0)]-\bar g_k^{T}(s)\right)\, \mathrm ds.
\end{equation}


Using \citet[Theorem VIII.3.79]{JacodShiryaev} with $p=q=2$ applied to arbitrary linear combinations of the components, in combination with the Cramér--Wold device, Lemma~\ref{cramerwold}, it follows that 
$$\frac1{\sqrt T\sqrt{\mu_{\hat\theta_T}^{(k)}}}\int_0^{uT}\left(g(s)-\mathbb E[g(0)]\right)\, \mathrm ds\stackrel d\longrightarrow W_2(u),$$ as $T\to\infty$, weakly on $D([0,\tau],\mathbb R^d)$,
where $W_2$ is a $d$-variate Wiener process with covariance given by (\ref{covW1}).

Next, by the same result it follows that $$\frac{(1-u)\sqrt T}{\sqrt{\mu_{\hat\theta_T}^{(k)}}}\left(\mathbb E[g_k(0)]-\bar g_k^{T}(uT)\right)\stackrel d\longrightarrow W_3^{(k)}(1-u),$$ as $T\to\infty$, weakly on $D([0,1],\mathbb R^d)$, 
for another Wiener process $W_3^{(k)}$ of variance (\ref{covW1}).
Apply the change of variables $s'=Ts$ to the second term of (\ref{Dkdecomp}). 
By the continuous mapping theorem \citep[Theorem 18.11]{vdVaart} applied to $h:D([0,\tau],\mathbb R^d)\to D([0,\tau],\mathbb R^d)$ given by $(h(f))^{(k)}(u)=\frac1{1-u}\int_0^{u}f^{(k)}(s)\, \mathrm ds$ it follows that the second term of (\ref{Dkdecomp}) converges weakly on $D([0,\tau],\mathbb R^d)$ equipped with the strong Skorokhod $J_1$ topology to $V_3$.
$\hfill\Box$
\newline

\textit{Proof of Theorem \ref{consistency thm}}.
Lemma \ref{thm5lemma} proves the convergence of $\hat W^{(T)}:=\mathscr T_{\hat\theta_T}(\hat\eta^{(T)})$ to $W_1+W_2+V_3$. 
Note that both $W_2$ and $V_3$ are non-zero stochastic processes if there is some $k\in[d]$ such that $\mathrm{Var}(g_k(0))\neq0$, hence if $\lambda$ is not equal (in $L^1(\mathbb P)$) to $\lambda_{\theta}+c$ for some $\theta\in\Theta$ and some $c\in\mathbb R^d$. 

Now consider the goodness-of-fit testing procedure outlined in Algorithm \ref{alg1}, selecting $n(T)$ such that $T/n(T)\to\infty$ and $n(T)\to\infty$ as $T\to\infty$, where we base the test on the transformed process  $\hat W^{(T)}:=\mathscr T_{\hat\theta_T}(\hat\eta^{(T)})$ under misspecification $\lambda\notin\mathscr L_\Theta=\{\lambda_\theta:\theta\in\Theta\}$. Consider step \emph{(iv)} of that testing procedure. The sample $(\hat Z_i^{(T)})_{i\in[n(T)]}$ converges in distribution to a sample of standard normal random vectors coming from $W_1$; plus a normally distributed perturbation coming from $W_2$; plus a  contribution coming from $V_3$. In order to detect the deviation from the null hypothesis, the contribution of $W_2+V_3$ should be visible in the sample $(\hat Z_i^{(T)})_{i\in[n(T)]}$, asymptotically, as the sample size $n(T)$ diverges, as $T\to\infty$.

First, consider the contribution of $V_3$ which is defined as the integral of a Wiener process $W_3$, i.e., $V_3(t)=\int_0^tW_3(1-s)\ \mathrm ds$. Note that, with $n=n(T)$,
$$
\left(\sqrt{\frac{n}{\tau}} \left(V_3(i\tau/n) - V_3((i-1)\tau/n)\right)\right)_{i\in[n]}=\left(\frac1{\sqrt {n\tau}}W_3((n-i)\tau/n)+\mathcal O_{\mathbb P}(1/n)\right)_{i\in[n]}
$$
composes a sample of size $n(T)$, with magnitude of order $1/\sqrt {n(T)}$, having a positive autocorrelation converging to $1$ at rate $n(T)$. 
Typical normality tests based on functionals of the test process, like the Kolmogorov-Smirnov, Cramér-von Mises and Anderson-Darling tests, are consistent at rate $\sqrt{n(T)}$. 
Hence, the term coming from $V_3$, which is of magnitude $1/\sqrt {n(T)}$, contributes non-trivially to the power of the test. 
However, the perturbation coming from $V_3$ does not guarantee the power of the test to converge to $1$.

Second, we consider the contribution of $W_2$. Note that $W_1$ is a standard Wiener process, while $W_2$ is a non-zero Wiener process. 
When condition (\ref{weirdcondition}) holds, $W_1+W_2$ does \emph{not} reduce to a process having variance $u\mapsto uI_d$. 
Hence, for any normality test in step \emph{(v)} of Algorithm \ref{alg1} that is consistent under alternatives (such as the Kolmogorov-Smirnov, Cramér-von Mises and Anderson-Darling tests), the deviation of $W_1+W_2$ from a standard Wiener process under the alternative $N$ is detected with probability $1$, asymptotically, as $T\to\infty$. 
$\hfill\Box$

\section*{Appendix C: Details relegated from Section~\ref{Section5.1}}\label{app C}

In this appendix, we first provide some intuitive explanation for the varying degree of difficulty in detecting different deviations from the exponential Hawkes null hypothesis $H_0^{\mathrm{Haw:E}}$ considered in Section~\ref{Section5.1}. 
Next, we provide rejection rates under the power-law Hawkes null hypothesis $H_0^{\mathrm{Haw:P}}$, and for the testing procedure described in Section~\ref{sectionRTC}.

We have seen in Section~\ref{Section5.1} that detecting deviations from the null $H_0^{\mathrm{Haw:E}}$ varies in difficulty. 
As it turns out, the law $\mathbb P^{\mathrm{Haw:P}}$ 
is challenging to distinguish from the null $\mathbb P^{\mathrm{Haw:E}}$, whereas detecting $\mathbb P^{\mathrm{SN}}$ 
is more feasible, and detecting $\mathbb P^{\mathrm{Periodic}}$ and $\mathbb P^{\mathrm{SC}}$ is relatively easy.
These differences can be explained as follows. 
First, $\mathbb P^{\mathrm{Haw:P}}$ is the law a self-exciting Hawkes process similar to the law under $H_0^{\mathrm{Haw:E}}$, but only with a power-law decay for the excitation kernel instead of an exponential decay. 
Detecting such a difference in kernel shape is inherently difficult --- in particular under the current values of the parameters that have deliberately been selected to make the two processes, $\mathbb P^{\mathrm{Haw:E}}$ and $\mathbb P^{\mathrm{Haw:P}}$, similar --- unless a (very) large sample is available (see below). 
Second, whereas the shot-noise law $\mathbb P^{\mathrm{SN}}$ 
exhibits clustering behavior like the Hawkes process, it lacks self-excitation. 
When its shot intensity $\mu$ 
is relatively high compared to $\beta$, 
clusters frequently overlap, making it harder to distinguish between exogenous and endogenous, self-exciting clustering.
By contrast, when $\mathbb P^{\mathrm{SN}}$ 
has a relatively low shot intensity, more clearly separated clusters result, and the distinction between self-excitation and short noise becomes more feasible. 
Next, consider the easier alternatives. 
The periodic Poisson process $\mathbb P^{\mathrm{Periodic}}$ produces peaks in intensity, leading to concentrated event groups that resemble clusters. 
For small $T$, detecting this may be challenging due to the limited number of observed ``clusters''. 
However, the regularity of these clusters will reveal that the process is not an overdispersed one. 
Finally, since the self-correcting process $\mathbb P^{\mathrm{SC}}$ exhibits behavior opposite to that of a Hawkes process, with each arrival \emph{decreasing} the conditional intensity instead of increasing it, rejecting the null under this alternative is relatively easy.

Although Theorem~\ref{consistency thm} ensures that the power of the Kolmogorov-Smirnov, Cramér-von Mises and Anderson-Darling tests converges to 1 as $T \to \infty$ under \emph{any} alternative, the previous paragraph suggests that, in practice (i.e., in finite samples), detecting the deviation from the null induced by $\mathbb P^{\mathrm{Haw:P}}$ 
requires a sample $\omega_T$ over a \emph{large} time frame $[0,T]$.
To investigate this, we run 100 computationally intensive simulations for $H_0^{\mathrm{Haw:E}}$ versus $N\sim\mathbb P^{\mathrm{Haw:P}}$ 
for a large value of $T$. 
The rejection rates for $N\sim\mathbb P^{\mathrm{Haw:P}}$ are well above the nominal significance levels at $T = 5 \cdot 10^5$. 
These rejection rates are summarized in Table~\ref{tablebigT}.

Next, we repeat our experiments for $H_0^{\mathrm{Haw:P}}$ consisting of univariate power-law Hawkes processes. 
We report our findings for $T=5{,}000$ and $T=50{,}000$ in Table~\ref{table3} and Table~\ref{table4}, respectively. 
The observed rejection rates are comparable to the ones in Table~\ref{table1} 
for $H_0^{\mathrm{Haw:E}}$. 
This confirms the results of Sections~\ref{Section3}--\ref{Section4.2} once more. 
Furthermore, it is possible to draw Q-Q plots analogous to Figure~\ref{fig1}, yielding comparable results.

\begin{remark}
Our results also hold for null hypotheses other than the self-exciting ones considered here. 
For example, for a null hypothesis consisting of the laws consisting of periodic Poisson processes $\mathbb P^{\mathrm{Periodic}}_{\mu,\alpha,\beta,\gamma}$, we have run $500$ simulations under $N\sim\mathbb P^{\mathrm{Periodic}}$, with $T=5{,}000$ and using the Anderson-Darling normality test. 
Using significance levels $0.01$, $0.05$ and $0.20$, this yields $10$, $32$ and $104$ rejections when using the transformation-based testing procedure, and $0$, $0$ and $5$ rejections when using the naive testing procedure, respectively.
\end{remark}

Note that the Kolmogorov-Smirnov, Cramér-von Mises and Anderson-Darling tests all yield goodness-of-fit tests of the right size when applied in step \emph{(v)} of the transformation-based testing procedure outlined in Algorithm~\ref{alg1}. 
Since in our simulations the power against alternatives is usually the highest for Anderson-Darling and the lowest for Kolmogorov-Smirnov, we advise to use the former in applications. 

We conclude by examining the random-time-change-based test discussed in Section~\ref{sectionRTC}, which, like the naive testing procedure in this section, ignores estimation uncertainty. 
We fit the same five processes as before to both $H_0^{\mathrm{Haw:E}}$ and $H_0^{\mathrm{Haw:P}}$, using the time frame $[0,T]$ with $T = 5{,}000$, and then perform the goodness-of-fit test outlined in Section~\ref{sectionRTC}. 
The results are summarized in Table~\ref{tableGoF}.  
We observe that the random-time-change-based test, similar to the naive procedure, appears undersized. 
Additionally, the power under any alternative is lower than that of our asymptotically exact test described in Section~\ref{Section4.1}. 
Thus, the random-time-change-based test shares the shortcomings of the naive testing procedure.

\begin{table}[H]\centering
\caption{{\small We report the percentage of rejections, 
out of $100$ tests, at significance levels $0.01; 0.05; 0.20$, in an experiment analogous to the one conducted for Table~\ref{table1}, using the parametric null hypothesis $H_0^{\mathrm{Haw:E}}$ and the alternative model $\mathbb P^{\mathrm{Haw:P}}$ 
for a large time frame with $T=5\cdot10^5$. 
}}
\bitsmall
\begin{tabular}{c | c c c | c c c}\label{tablebigT}
	&\multicolumn{3}{c}{Transformation-based testing procedure}&\multicolumn{3}{c}{Naive testing procedure}\\
Test&KS&CvM&AD&KS&CvM&AD\\
\hline
$\mathbb P^{\mathrm{Haw:P}}$			&3\%; 10\%; 35\%	&2\%; 16\%; 40\%	&3\%; 20\%; 41\%	&0\%; 3\%; 11\%	&0\%; 1\%; 13\%	&0\%; 5\%; 19\%	\\
\end{tabular}
\normalsize
\end{table}

\begin{table}[H]\centering
\caption{{\small We report the number of rejections, $R_{0.01}; R_{0.05}; R_{0.20}$, out of $500$ tests, at significance levels $0.01; 0.05; 0.20$, in an experiment analogous to the one conducted for Table~\ref{table1}, but now with parametric null hypothesis $H_0^{\mathrm{Haw:P}}$ and $T=5{,}000$.}}
\bitsmall
\begin{tabular}{c | c c c | c c c}\label{table3}
	&\multicolumn{3}{c}{Transformation-based testing procedure}&\multicolumn{3}{c}{Naive testing procedure}\\
Test&KS&CvM&AD&KS&CvM&AD\\
\hline
$\mathbb P^{\mathrm{Haw:P}}$			&3; 23; 100	&4; 20; 97	&6; 20; 100	&0; 0; 9	&0; 0; 3	&0; 0; 10	\\
$\mathbb P^{\mathrm{Haw:E}}$			&7; 22; 92	&6; 18; 90	&4; 14; 86	&0; 0; 4	&0; 0; 2	&0; 0; 3	\\
$\mathbb P^{\mathrm{SN}}$ &8; 47; 243	&4; 43; 298	&1; 41; 304	&0; 5; 140	&0; 11; 204	&0; 14; 245	\\
$\mathbb P^{\mathrm{Periodic}}$		&15; 131; 417	&10; 169; 460	&3; 173; 475	&0; 46; 350	&1; 98; 431	&0; 111; 451	\\
$\mathbb P^{\mathrm{SC}}$			&500; 500; 500	&500; 500; 500	&500; 500; 500	&500; 500; 500	&500; 500; 500	&500; 500; 500	\\
\end{tabular}
\normalsize
\end{table}

\begin{table}[H]\centering
\caption{{\small We report the number of rejections, $R_{0.01}; R_{0.05}; R_{0.20}$, out of $500$ tests, at significance levels $0.01; 0.05; 0.20$, in an experiment analogous to the one conducted for Table~\ref{table1}, but now with parametric null hypothesis $H_0^{\mathrm{Haw:P}}$ and $T=50{,}000$.}}
\bitsmall
\begin{tabular}{c | c c c | c c c}\label{table4}
	&\multicolumn{3}{c}{Transformation-based testing procedure}&\multicolumn{3}{c}{Naive testing procedure}\\
Test&KS&CvM&AD&KS&CvM&AD\\
\hline
$\mathbb P^{\mathrm{Haw:P}}$			&4; 24; 99	&2; 19; 101	&2; 18; 100	&0; 0; 4	&0; 0; 3	&0; 0; 6	\\
$\mathbb P^{\mathrm{Haw:E}}$			&3; 28; 115	&4; 24; 110	&2; 21; 110	&0; 0; 10	&0; 0; 9	&0; 0; 9	\\
$\mathbb P^{\mathrm{SN}}$ &104; 334; 489	&123; 412; 498	&244; 469; 500	&43; 269; 481	&87; 378; 496	&204; 468; 498	\\
$\mathbb P^{\mathrm{Periodic}}$		&431; 498; 500	&487; 500; 500	&499; 500; 500	&409; 499; 500	&493; 500; 500	&499; 500; 500	\\
$\mathbb P^{\mathrm{SC}}$			&500; 500; 500	&500; 500; 500	&500; 500; 500	&500; 500; 500	&500; 500; 500	&500; 500; 500	\\
\end{tabular}
\normalsize
\end{table}

\begin{table}[H]\centering
\caption{{\small We report the number of rejections, $R_{0.01}; R_{0.05}; R_{0.20}$, out of $500$ tests, at significance levels $0.01; 0.05; 0.20$, in an experiment similar to the one conducted for Table~\ref{table1}, but now using the random-time-change-based testing procedure described in Section~\ref{sectionRTC}, with $T=5{,}000$.}}
\bitsmall
\begin{tabular}{c | c c c | c c c}\label{tableGoF}
	&\multicolumn{3}{c}{$H_0^{\mathrm{Haw:E}}$}&\multicolumn{3}{c}{$H_0^{\mathrm{Haw:P}}$}\\
Test&KS&CvM&AD&KS&CvM&AD\\
\hline
$\mathbb P^{\mathrm{Haw:E}}$			&0; 0; 3	&0; 0; 0	&0; 0; 1	&0; 5; 91	&0; 1; 78	&0; 9; 121	\\\
$\mathbb P^{\mathrm{Haw:P}}$			&0; 4; 49	&0; 0; 36	&0; 3; 59	&0; 0; 1	&0; 0; 1	&0; 0; 2	\\
$\mathbb P^{\mathrm{SN}}$ &3; 66; 239	&0; 13; 212	&5; 51; 271	&0; 0; 7	&0; 0; 2	&0; 1; 9	\\
$\mathbb P^{\mathrm{Periodic}}$		&44; 165; 349	&26; 163; 391	&43; 193; 406	&500; 500; 500	&500; 500; 500	&500; 500; 500	\\
$\mathbb P^{\mathrm{SC}}$			&500; 500; 500	&500; 500; 500	&500; 500; 500	&500; 500; 500	&500; 500; 500	&500; 500; 500	\\
\end{tabular}
\normalsize
\end{table}

\newpage

\section*{Appendix D: Details relegated from Section~\ref{Section5.2}}\label{app D}
In Section~\ref{Section5.2}, we gave an overview of the analysis of two real-world data sets to which we applied our novel testing procedure. 
In this appendix, we provide the remaining details.

\subsection*{Appendix D.1\quad Temporal ETAS model}
First, we apply the testing procedure from Algorithm~\ref{alg1} to earthquake data analyzed already in \citet{Ogata2} --- the seminal paper that has established a framework for applying Omori's law and the Gutenberg-Richter law to earthquake data.\footnote{ 
Building on this temporal ETAS (Epidemic Type Aftershock Sequence) model, subsequent years have seen the development of the more advanced spatiotemporal ETAS model, which has become a cornerstone in earthquake modeling. 
For further developments and applications, we refer, e.g., to \citet{Ogata3, Veen, Adelfio, Reinhart}.}

We consider the earthquake data set used in \citet[Table 1]{Ogata2}, consisting of shallow earthquakes of less than $100$ kilometers depth, of magnitude at least $6.0$ on the Richter scale, in the Off Tohoku district east of Japan. 
This is the polygonal region with vertices at  ($42$\degree N,$142$\degree E), ($39$\degree N, $142$\degree E), ($38$\degree N, $141$\degree E), ($35$\degree N, $140.5$\degree E), ($35$\degree N, $144$\degree E), ($39$\degree N, $146$\degree E), and ($42$\degree N, $146$\degree E); see \citet[Fig.\ 1]{Ogata2} for a map. 
This region is one of the most seismically active areas in Japan. 
The data set spans the period between $1885$ and $1980$, and consists of $483$ earthquakes satisfying the depth and magnitude conditions. 

The earthquake dynamics are modeled by marked point processes, with time stamps $t_i$ corresponding to the time of the earthquake occurrences measured in days after January 1st, 1885, and marks $m_i$ corresponding to earthquake magnitudes measured on the Richter scale, assumed to be independent of $t_i$.
Since the data are recorded with a daily temporal resolution, the onset time for each earthquake was randomly assigned from a uniform distribution within the corresponding day.
More specifically, \citet{Ogata2} considers marked Hawkes processes with conditional intensity functions
\begin{equation}\label{Hawkesgeneral}
\lambda(t)=\mu+\sum_{t_i<t}c(m_i)g(t-t_i).
\end{equation}
Based on the Akaike information criterion, the model with $g(t)=K/(t+c)$ and $c(m)=e^{\beta(m-M)}$ is selected. 
Here, $M=6.0$ is the cut-off magnitude. 
The decay kernel $g(t)$ and the mark function $c(m)$ are theoretically motivated by Omori's law and the Gutenberg-Richter law, respectively. 
This yields the conditional intensity function 
\begin{equation}\label{HawkesOmoriApp}
\lambda_\theta(t)=\lambda_{\mu,K,c,\beta}(t)=\mu+\sum_{t_i<t}e^{\beta(m_i-M)}\frac K{t-t_i+c}.
\end{equation}
From \citet[Table 3]{Ogata2}, the maximum likelihood estimators for the parameters are 
\[
\hat\theta=(\hat\mu,\hat K,\hat c,\hat\beta)=(.00536,.017284,.01959,1.61385).
\] 

Having selected this ``best'' model, it is noted that it is still possible that there exists a more suitable model. 
To assess the suitability of the model \eqref{HawkesOmoriApp}, \citet{Ogata2} considers the transformed event times $(\tau_i)_{i\in[483]}$, where $\tau_i=\Lambda_{\hat\theta}(t_i)$, and where $\Lambda_\theta(t)=\int_0^t\lambda_\theta(s)\, \mathrm ds$ is the compensator of the point process. 
By the random time change theorem \citep[Theorem 7.4.I]{DaleyVereJones}, it follows that the transformed times $(\Lambda_{\theta_0}(t_i))_{i\in[483]}$, using the \emph{true} parameter $\theta_0$, form a realization of a stationary Poisson process of unit intensity. 
Therefore, it is noted in \citet{Ogata2} that if $\lambda_{\hat\theta}$ is a good approximation of $\lambda_{\theta_0}$, then  $(\tau_i)_{i\in[483]}$ is expected to behave like a stationary Poisson process of unit intensity. 
However, from Section~\ref{sectionRTC} we know that even if $\lambda_{\hat\theta}$ is a good approximation of $\lambda_{\theta_0}$, the difference between $(\Lambda_{\hat\theta}(t_i))_{i\in[483]}$ and $(\Lambda_{\theta_0}(t_i))_{i\in[483]}$ is not in general negligible: ignoring estimation uncertainty can lead to inaccurate tests. 
 
In \citet{Ogata2}, the null hypothesis that  $(\tau_i)_{i\in[483]}$ is a realization of a Poisson process of unit intensity is tested with the aid of a Kolmogorov-Smirnov test; see \citet{Ogata2}, Figure 9. 
One can use the transformed times $(\tau_i)_{i\in[483]}$ to construct the interarrival times $(x_i)_{i\in[482]}$, where $x_i:=\tau_{i+1}-\tau_i$ are standard exponentially distributed under the null. 
Following \citet{Ogata2}, we perform a Kolmogorov-Smirnov test for the null hypothesis stating that $(x_i)_{i\in[482]}$ is a sample from a standard exponential distribution. This yields a $p$-value of $.6380$. Hence, for any reasonable significance level, the null is not rejected, indicating a seemingly good fit of the model to the data.
 
We examine the situation in greater detail, accounting for the estimation uncertainty in $\hat\theta$ by performing our asymptotically exact transformation-based test described in Algorithm~\ref{alg1}. 
Specifically, the transformed empirical process is calculated using \eqref{transformation eq} as outlined in step \emph{(iii)} of the procedure, and its path is shown in Figure~\ref{fig2}. 
Under the null hypothesis that the true intensity satisfies \eqref{HawkesOmoriApp}, this figure should resemble a standard Brownian motion.  
Performing the transformation-based testing procedure involves selecting a hyperparameter $n$ and a normality test for step \emph{(v)} of the procedure. 
Unlike the setup in Section~\ref{Section4.1}, where there is approximately one event per time unit (days), here the choice $n=\mathrm{ceil}\left(\sqrt{N[0,T]}/4\right)=6$ instead of $n=\mathrm{ceil}(\sqrt{T}/4)$ is more appropriate. 
Using the Kolmogorov-Smirnov normality test, 
we obtain a $p$-value of $0.1076$ (and $0.7684$ for the naive testing procedure). 
The Anderson-Darling normality test yields a $p$-value of $.0568$.

The relatively low $p$-value and the unexpected jumps in Figure~\ref{fig2} raise concerns about the adequacy of the model's fit. 
While the $p$-value does not lead to rejection at a typical significance level such as $\alpha = 0.05$, it suggests that the model may not fully capture the underlying dynamics of the data. 
This is consistent with the fact that the temporal ETAS model may overlook certain features better accounted for in the spatiotemporal version. 
As we have seen, a test based on random time change, ignoring estimation uncertainty, yields a much higher $p$-value, indicating a seemingly good fit. 
These discrepancies highlight that ignoring estimation uncertainty can lead to inaccurate procedures with unwarranted confidence in an unfit model.


\subsection*{Appendix D.2\quad A recursive point process model}
More recently, a variant of the Hawkes process was introduced in \citet{Schoenberg}, having the property that the expected offspring size corresponding to an event is inversely related to the conditional intensity at the time that the event occurred. 
This ``recursive point process model'' is based on the idea that when a disease occurs infrequently within a population, as seen in the initial phases of an outbreak, individuals with the disease are likely to transmit it at a higher rate. 
By contrast, as the disease becomes more widespread, the transmission rate diminishes due to the growing proportion of individuals with prior exposure and the implementation of preventive measures. 
The model is described as recursive because the conditional intensity at any given moment in time depends on the productivity of earlier points, which is itself  influenced by the conditional intensity at those same points.

The recursive point process model is specified through its conditional intensity
\begin{equation}\label{recursive intensity}
\lambda(t)=\mu+\sum_{t_i<t}H(\lambda(t_i))g(t-t_i),
\end{equation}
where $\mu>0$, $H:(0,\infty)\to[0,\infty)$ is a non-increasing function and $g:[0,\infty)\to[0,\infty)$ is a density function. 
Note that this model is self-exciting, and can be seen as a variant of the classical Hawkes process. 
A parametric specification is chosen where $H$ satisfies a power law and where $g$ is an exponential density. 
More specifically, $H(x)=\kappa/x^\alpha$ and $g(u)=\beta\exp(-\beta u)$, yielding a parametric model depending on $\theta=(\mu,\kappa,\beta,\alpha)$, where all parameters are assumed to be positive. 

In \citet[Section 7]{Schoenberg}, this model is fitted to data consisting of $67$ recorded cases of Rocky Mountain Spotted Fever (RMSF) in California from 1960 to 2011.
The data consist of time stamps marking the beginning of each week during which at least one infection occurred. 
Time is recorded in days since January 1st, 1960.
The freely accessible data are obtained from Project Tycho (www.tycho.pitt.edu); see \citet{VanPanhuis2018}. 

We replicate the model estimation from \citet[Section 7]{Schoenberg}.
As in \citet{Schoenberg}, since the data are recorded with a weekly temporal resolution, we randomly assign the onset time for each individual case from a uniform distribution within the corresponding 7-day interval. 
We obtain the following parameter estimates for the California data set:\footnote{A close inspection of the data set obtained from Project Tycho, in comparison with the histogram presented in \citet[Figure 2]{Schoenberg}, reveals a minor discrepancy between the two data sets, which cannot be attributed to a different realization of the uniform distribution within the corresponding 7-day interval. 
It is likely due to a correction in the data, and may account for the difference between our parameter estimates and those reported by Schoenberg.}
\begin{equation}
(\hat\mu_{CA},\hat\kappa_{CA},\hat\beta_{CA},\hat\alpha_{CA})=(0.000415,0.0775,0.00562,0.473).
\end{equation}

Additionally, we collect data for the state of Florida over the same time span, from 1960 to 2011, consisting of $191$ recorded disease cases, not considered in \citet{Schoenberg}. 
Estimating the model from \citet{Schoenberg} to the Florida data set yields the following parameter estimates:
\begin{equation}
(\hat\mu_{FL},\hat\kappa_{FL},\hat\beta_{FL},\hat\alpha_{FL})=(0.000777,0.833,0.0110,0.0408).
\end{equation}

In \citet[Section 7]{Schoenberg}, the model fit to the California data is assessed using super-thinned residuals. 
This approach is based on principles closely related to those underlying the random-time-change-based test described in Section~\ref{sectionRTC}. 
Notably, both methods rely on plugged-in estimates for the model parameters and do not account for estimation uncertainty, which can compromise the validity of the resulting inference.
Instead of using super-thinned residuals, we apply a classical goodness-of-fit test based on the random time change, as outlined in Section~\ref{sectionRTC}. 
After fitting the model, we transform the interarrival times $t_{i+1}-t_i$ using the estimated compensator, yielding the transformed values $\Lambda_{\hat\theta}(t_{i+1})-\Lambda_{\hat\theta}(t_i)$. 
We then test whether these transformed interarrival times follow a standard exponential distribution. 
The corresponding QQ-plots, shown in Figure~\ref{fig3}, include the reference line $y=x$, around which the transformed interarrival times should be centered under the null hypothesis. 
For the California data set, a Kolmogorov-Smirnov test yields a $p$-value of $0.3184$, which aligns with the findings in \citet[Figure 3]{Schoenberg}. 
While there is a slight deviation from the reference line $y=x$ in Figure~\ref{fig3}(A), it remains within the Kolmogorov-Smirnov confidence bounds, and we therefore do not reject the null hypothesis at the $0.05$ significance level. 
For the Florida data set, the deviation from the reference line appears more pronounced than for the California data; nonetheless, the Kolmogorov-Smirnov test still does not reject the null hypothesis, with a $p$-value of $0.4836$.

\begin{figure}[ht]
\begin{tabular}{cc}
\subfloat[California data set]{\includegraphics[width = 3.2in]{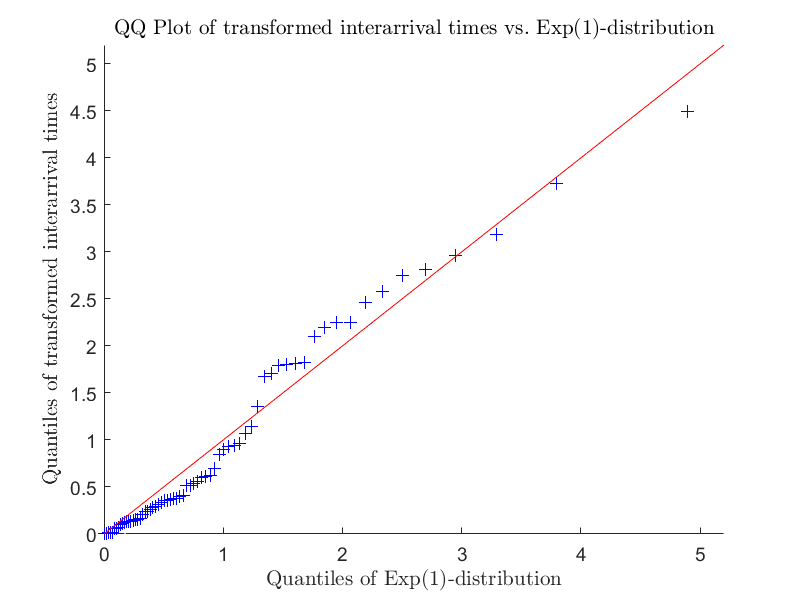}} &
\subfloat[Florida data set]{\includegraphics[width = 3.2in]{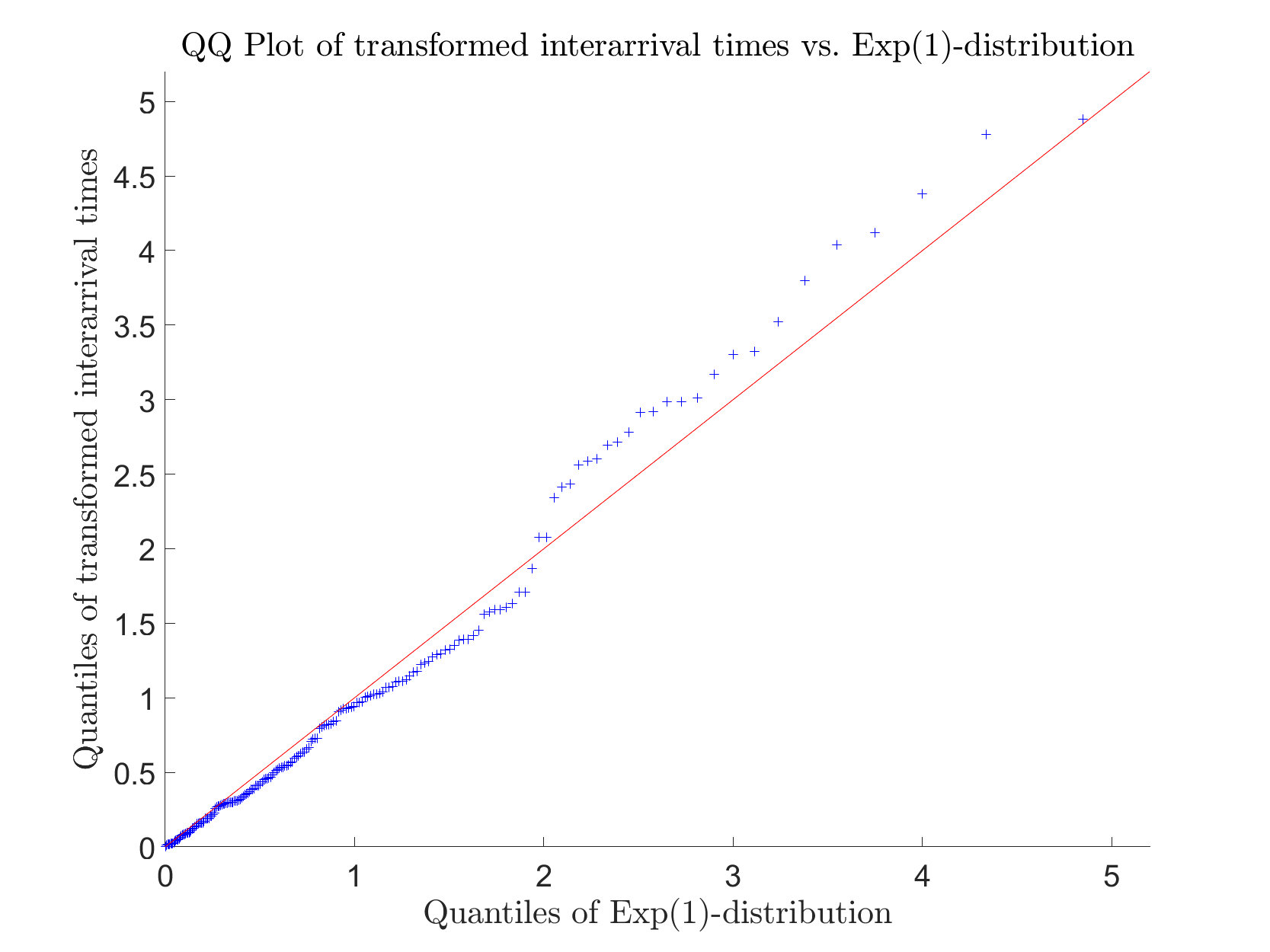}}
\end{tabular}
\caption{\small{\label{fig3} For both the California (a) and the Florida (b) data sets we estimate the model to the data, after which we transform arrival times and test for standard exponentiality of the interarrival times. 
There appear to be slight deviations from $y=x$, especially for the tail of the Florida data set.}}
\end{figure} 

Next, we account for the estimation uncertainty in $\hat\theta$ by applying our asymptotically exact transformation-based test described in Algorithm~\ref{alg1}. 
We set the hyperparameter $n=\max\left\{\mathrm{ceil}\left(\sqrt{N[0,T]}/4\right),6\right\}=6$ for both the California data ($N[0,T]=67$) and the Florida data ($N[0,T]=191$). 
We perform an Anderson-Darling test for normality in step \emph{(v)}, as suggested in Section~\ref{Section5.1}.

For the California sample, this test yields a $p$-value of $0.5193$ (compared to $0.8833$ for the naive testing procedure). 
These results are consistent with the conclusions from applying the random-time-change-based tests above, at the $\alpha=0.05$ level. 
By contrast, for the Florida sample, the $p$-value is $0.0185$ (compared to $0.1979$ for the naive testing procedure), leading us to reject the null hypothesis when accounting for estimation uncertainty. 
In fact, those conclusions remain to hold if instead of an Anderson-Darling test we apply Kolmogorov-Smirnov or Cramér-von Mises tests. 
See Table~\ref{pvaluestab}.

Hence, whereas for the California data the asymptotically exact testing procedure in Algorithm~\ref{alg1} leads to the same conclusion as the procedures that ignore estimation uncertainty, the outcome is notably different for the Florida data: Algorithm~\ref{alg1} leads to a rejection of the null hypothesis, whereas the naive and random-time-change-based procedures do not. 
As discussed in Section~\ref{Section5.1}, tests that ignore estimation uncertainty are typically undersized and have limited power against alternatives. 
While the model appears to fit the California data reasonably well --- even though the sample size $N[0,T]=67$ is quite modest --- the seemingly adequate fit suggested by the naive and random-time-change-based tests for the Florida data is refuted by our asymptotically exact procedure. 
In typical settings where a new model is being evaluated, neglecting estimation uncertainty can result in misleading conclusions and overconfidence in an unfit model. 
The contrasting outcomes for the Florida data set highlight the importance of accounting for estimation uncertainty to avoid erroneous inference.

\end{document}